\documentclass{article}

\usepackage[utf8]{inputenc} 
\usepackage{amsfonts}
\usepackage{amsmath}
\usepackage{amsthm} 
\usepackage{color}
\usepackage{graphicx}
\usepackage{setspace}
\usepackage{enumerate}
\usepackage{booktabs}
\usepackage{subfigure}
\usepackage{pstricks}
\usepackage{url}
\usepackage{hyperref}
\usepackage{mathtools}
\usepackage{multirow}
\usepackage{enumitem}
\usepackage{cancel}
\usepackage[ruled, boxed, noend, noline, linesnumbered, norelsize]{algorithm2e}
\usepackage[normalem]{ulem}
\hypersetup{
	colorlinks,
	linkcolor={red!50!black},
	citecolor={blue!50!black},
	urlcolor={blue!80!black}
}

\definecolor{myred}{RGB}{255,50,50}         
\definecolor{myblack}{RGB}{0,0,0}           

\newtheorem{theorem}{Theorem}[section]

\newtheorem{lemma}[theorem]{Lemma}
\newtheorem{definition}[theorem]{Definition}
\newtheorem{corollary}[theorem]{Corollary}
\newtheorem{proposition}[theorem]{Proposition}

\numberwithin{equation}{section}

\newcommand{\argmin}{\mathop{\mathrm{argmin}}}   
\newcommand{\grad}{\nabla}                       
\renewcommand{\implies}{\Rightarrow}             
\newcommand{\inner}[2]{\langle#1,#2\rangle}      
\newcommand{\interior}{\mathrm{int}\,}           
\newcommand{\norm}[1]{\left\|#1\right\|}                    
\renewcommand{\Re}{\mathbb{R}}                   
\newcommand{\tr}{\mathrm{tr}}                    
\newcommand{\jProd}[2]{ {#1 \circ #2 } }		 
\newcommand{\PSDcone}[1]{{\mathcal{S}^{#1}_+}}	 
\newcommand{\tanCone}[2]{ {\mathcal{T}_{#2}(#1)}}	 
\newcommand{\matRank}{{\mathrm{ rank } \,}}	
\newcommand{\lineality}{\mathrm{lin}\,}

\renewcommand{\doteq}{:=}
\newcommand{\stdCone}{ {\mathcal{K}}}
\newcommand{\jAlg}{\mathcal{E}}
\newcommand{\aug}[1]{\mathfrak{L}_{#1}}
\newcommand{\augP}{\rho}
\newcommand{\mult}{\mu}
\newcommand{\eig}{\sigma}
\newcommand{\projC}[1]{P_{#1}}
\newcommand{\proj}[1]{\left[#1\right]_+}
\newcommand{\augSlack}{\aug{\augP}^{\text{Slack}}}
\newcommand{\augSlackP}[1]{\aug{#1}^{\text{Slack}}}
\newcommand{\augCone}{\aug{\augP}^{\text{Sym}}}
\newcommand{\augConeP}[1]{\aug{#1}^{\text{Sym}}}

\newcommand{\stdFace}{ {\mathcal{F}}}

\newcommand{\minFacePoint}[2]{ {\mathcal{F}(#1,#2)}}
\DeclarePairedDelimiter\abs{\lvert}{\rvert}%
\newcommand{\ball}[2]{ {B_{#2}(#1)}}
\newcommand{\jac}[1]{J#1}

\begin{document}
\title{Optimality conditions for problems over symmetric cones\\ and a simple augmented Lagrangian method%
	}
\author{
	Bruno F. Louren\c{c}o%
	\thanks{Department of Computer and Information Science, 
		Faculty of Science and Technology, 
		Seikei University, Tokyo 180-8633, Japan
		(\texttt{lourenco@st.seikei.ac.jp}).}
	\and   
	Ellen H. Fukuda%
	\thanks{Department of Applied Mathematics and Physics, 
		Graduate School of Informatics, 
		Kyoto University, Kyoto 606-8501, Japan
		(\texttt{ellen@i.kyoto-u.ac.jp}).}
	\and 
	Masao Fukushima%
	\thanks{
		Department of Systems and Mathematical Science, 
		Faculty of Science and Engineering,
		Nanzan University, Nagoya, Aichi 466-8673, Japan
		(\texttt{fuku@nanzan-u.ac.jp}).
	}
}
\date{September 1, 2017 (Updated \today)}
\maketitle
\begin{abstract}
\noindent In this work we are interested in nonlinear symmetric cone problems (NSCPs), 
which contain as special cases nonlinear semidefinite programming, nonlinear second order cone 
programming and the classical nonlinear programming problems. We
explore the possibility of reformulating NSCPs as common nonlinear programs (NLPs), 
with the aid of squared slack variables.
Through this connection, we show how to obtain second order optimality 
conditions  for NSCPs in an easy manner, thus bypassing a number of 
difficulties associated to the usual variational analytical approach. 
We then discuss several aspects of this connection. In particular, 
we show a ``sharp'' criterion for membership in a symmetric cone that also encodes rank information.
Also, we discuss the possibility of importing convergence results from  
nonlinear programming to NSCPs, which we illustrate by discussing a simple augmented Lagrangian method for nonlinear symmetric cones. 
We show that, employing the slack variable approach, we can use the results 
for NLPs to prove convergence results, thus extending
a special case (i.e., the case with strict complementarity) of an earlier result by Sun, Sun and Zhang for nonlinear semidefinite programs.\\

\end{abstract}

\section{Introduction}
In this paper, we  analyze optimality conditions for the following problem:
\begin{equation}
  \label{eq:sym}
  \tag{P1}
  \begin{array}{ll}
    \underset{x}{\mbox{minimize}} & f(x) \\ 
    \mbox{subject to} &  h(x) = 0,\\
    & g(x) \in \stdCone,\\
    & x \in \Re^n,
  \end{array}
\end{equation}
where $f \colon \Re^n \to \Re, h \colon \Re^n \to \Re^m$ and $g \colon \Re^n \to \jAlg$ are twice
continuously differentiable functions. We assume that $\jAlg$ is a 
finite dimensional space equipped with an inner product $\inner{\cdot}{\cdot}$.
Here,  $\stdCone \subseteq \jAlg$ is a \emph{symmetric cone}, 
that is
\begin{enumerate}
\item $\stdCone$ is \emph{self-dual}, \emph{i.e.}, $\stdCone \coloneqq \stdCone ^* = \{x \mid \inner{x}{y} \geq 0, \forall y \in \stdCone \},$

\item $\stdCone$ is \emph{full-dimensional}, i.e., the interior of $\stdCone$ is 
not empty,
\item $\stdCone$ is \emph{homogeneous}, i.e., for every pair of points $x,y$ in the interior of $\stdCone$, there is a linear bijection $T$ such that $T(x) = y$ and $T(\stdCone) = \stdCone$. In short, 
the group of automorphism of $\stdCone$ acts transitively on the interior of $\stdCone$.
\end{enumerate}

Then, problem \eqref{eq:sym} is called a \emph{nonlinear symmetric cone problem} (NSCP).
One of the advantages of dealing with NSCPs is that it provides an unified framework for a number 
of different problems including classical nonlinear programs (NLPs), nonlinear semidefinite programs (NSDPs), 
nonlinear second order cone programs (NSOCPs) and any mixture of those three.

Since $\stdCone$ is a symmetric cone, we may assume that $\jAlg$ is equipped with a bilinear map $\jProd{}{}:\jAlg\times \jAlg \to \jAlg$ such that 
$\stdCone$ is the corresponding cone of squares, that is,
$$
\stdCone = \{\jProd{y}{y} \mid y \in \jAlg\}.
$$ 
Furthermore, we assume that $\jProd{}{}$ possesses the following three properties:
\begin{enumerate}
	\item $\jProd{y}{z} = \jProd{z}{y}$,
	\item $\jProd{y}({\jProd{y^2}{z}}) = \jProd{y^2}({\jProd{y}{z}})$, where $y^2 = \jProd{y}{y}$,
	\item $\inner{\jProd{y}{z}}{w} = \inner{y}{\jProd{z}{w}}$,
\end{enumerate}
for all $y,w,z \in \jAlg$. Under these conditions, $(\jAlg, \jProd{}{})$ is called 
an \emph{Euclidean Jordan algebra}. It can be shown that every symmetric cone arises 
as the cone of squares of some Euclidean Jordan algebra,  see Theorems III.2.1 and III.3.1 in \cite{FK94}. 
We will use $\norm{\cdot}$ to indicate the norm induced by $\inner{\cdot}{\cdot}$. We remark that no 
previous knowledge on Jordan algebras will be assumed and we will try to be as self-contained as possible.

While the analysis of first order conditions for \eqref{eq:sym} is relatively straightforward, 
it is challenging to obtain a workable description of second order conditions  for \eqref{eq:sym}.
We recall that for NSOCPs and NSDPs, in order to obtain the so-called ``zero-gap'' optimality conditions, 
it is not enough to build the Lagrangian and require it to be positive semidefinite/definite over the critical 
cone. In fact, an extra term is needed which, in the literature, is known as the \emph{sigma-term} and is said to model the curvature 
of the underlying cone. The term ``zero-gap'' alludes to the fact that the change from ``necessary'' to ``sufficient'' should 
involve, apart from minor technicalities,  only a change from ``$\geq$'' to ``$>$'',
as it is the case for classical nonlinear programming (see Theorems 12.5 and 12.6 in \cite{NW99} or Section 3.3 in \cite{Ber99}).   

Typically, there are two approaches for obtaining ``zero-gap'' second order conditions. 
The first is to compute directly the so-called 
\emph{second order tangent sets} of $\stdCone$. This was done, for instance, 
by Bonnans and Ram\'irez in \cite{BR05} for NSOCP. Another approach is 
to express the cone as
$$
\stdCone = \{ z \in \jAlg \mid \varphi(z) \leq 0\}, 
$$
where $\varphi$ is some convex function. Then, the second order tangent 
sets can be computed by examining the second order directional derivatives of 
$\varphi$. This is the approach favored by Shapiro in \cite{shapiro97} for NSDPs.

For $\stdCone$, there is a natural candidate for $\varphi$. 
Over an Euclidean Jordan algebra, we have a ``minimum eigenvalue 
function'' $\eig _{\min}$, for which $x \in \stdCone$ if and only if $\eig _{\min}(x) \geq 0$, in analogy 
to the positive semidefinite cone case. 
We then take  $\varphi = - \eig _{\min}$. 
Unfortunately, as far as we know, it is still an 
open problem to give explicit descriptions of higher order directional derivatives for 
$-\eig _{\min}$. 
In addition, it seems complicated to describe the second order 
tangent sets of $\stdCone$ directly. 

In fact, the only reference we could find that discussed second order 
conditions for NSCPs is the work by Kong, Tun\c{c}el and Xiu \cite{Kong2011}, where they define the strong 
second order sufficient condition for a linear symmetric cone program, see Definition 3.3 therein.
Here, we bypass all these difficulties by exploring the Jordan algebraic connection and 
transforming \eqref{eq:sym} into an ordinary nonlinear program with equality constraints:
\begin{equation}
  \label{eq:slack}
  \tag{P2}
  \begin{array}{ll}
    \underset{x,y}{\mbox{minimize}} & f(x) \\ 
    \mbox{subject to} & h(x) = 0,\\ 
    & g(x) =  \jProd{y}{y},\\
    & x\in \Re^n, y \in \jAlg.
  \end{array}
\end{equation}
We will then use \eqref{eq:slack} to derive optimality conditions for 
\eqref{eq:sym}. By writing down the second order conditions  for 
\eqref{eq:slack} and eliminating the slack variable $y$, we can obtain 
second order conditions for \eqref{eq:sym}. This is explained in more 
detail in Section \ref{sec:soc}. The drawback of this approach is that 
the resulting second order conditions require strict complementarity. 
How serious this drawback is depends, of course, on the specific applications
one has in mind. Still, we believe the connection between the two formulations 
can bring some new insights.

In particular, through this work we found a ``sharp'' characterization of 
membership in a symmetric cone. Note that since 
$\stdCone$ is self-dual, a necessary and sufficient condition for some $\lambda \in \jAlg$ to belong to $\stdCone$ is that $\inner{z}{\lambda} \geq 0$ holds for all $z \in \stdCone$, 
or equivalently, that $\inner{\jProd{w}{w}}{\lambda} \geq 0$ for all $w \in \jAlg$. 
This, however, gives 
no information on the rank of $\lambda$. In contrast, Theorem~\ref{theo:strict} shows that if we 
instead require that $\inner{\jProd{w}{w}}{\lambda} > 0$ for all nonzero $w$ in some 
special subspace of $\jAlg$, this  not only guarantees that $\lambda \in \stdCone$, but 
also reveals information about its rank. This generalizes Lemma~1 in \cite{LFF16} for 
all symmetric cones.

Moreover, our analysis opens up the possibility of importing convergence results from 
the NLP world to the NSCP world, instead of proving them from scratch. 
In Section \ref{sec:alg}, 
we illustrate this by extending a result of Sun, Sun and Zhang \cite{SSZ08} on the quadratic augmented Lagrangian method. 

The paper is organized as follows. In Section~\ref{sec:prel}, 
we review basic notions related to Euclidean Jordan algebras, KKT points 
and second order conditions. In
Section~\ref{sec:crit}, we prove a criterion for membership in a 
symmetric cone. In Section \ref{sec:kkt}, we provide sufficient conditions 
that guarantee equivalence between KKT points of \eqref{eq:sym} and 
\eqref{eq:slack}.
In Section~\ref{sec:cons}, we discussion the relation between constraint
qualifications of those two problems. In Section \ref{sec:soc}, we present 
second order conditions for \eqref{eq:sym}. In Section~\ref{sec:alg}, we discuss 
a  simple augmented Lagrangian method. We conclude in Section~\ref{sec:conc}, with final remarks and a few suggestions
for further work.

%

\section{Preliminaries}\label{sec:prel}
\subsection{Euclidean Jordan algebra}
We first review a few aspects of the theory of Euclidean Jordan algebras. More 
details can be found in the book by Faraut and Kor\'anyi \cite{FK94} and also 
in the survey paper by Faybusovich \cite{FB08}. 
First of all, any 
symmetric cone $\stdCone$ arises as the cone of squares of some Euclidean Jordan algebra $(\jAlg, \jProd{}{})$. Furthermore, 
we can assume that $\jAlg$ has an unit element $e$ satisfying $\jProd{y}{e} = y$, for all $y \in \jAlg$. 
Reciprocally, given an Euclidean Jordan algebra $(\jAlg, \jProd{}{})$, it can be shown that the corresponding 
cone of squares is a symmetric cone. See Theorems III.2.1 and III.3.1 in \cite{FK94}, for more details.

Given $y \in \jAlg$, we denote by $L_y$ the linear operator such that
$$
L_y(w) = \jProd{y}{w},
$$
for all $w \in \jAlg$.

In what follows, we say that $c$ is an \emph{idempotent} if 
$\jProd{c}{c} = c$. Morover, $c$ is \emph{primitive} if it is nonzero and there is no way of writing 
$
c = a+b,
$
with  nonzero idempotents $a$ and $b$ satisfying $\jProd{a}{b} = 0$.

\begin{theorem} [Spectral Theorem, see Theorem III.1.2 in \cite{FK94}]\label{theo:spec}
	Let $(\jAlg, \jProd{}{} )$ be an Euclidean Jordan algebra and let $y \in \jAlg$. Then there are	 primitive idempotents $c_1, \dots, c_r$ satisfying
		\begin{flalign}
		& \jProd{c_i}{c_j} = 0 \, \,\,\,\, \qquad \qquad \text{ for } i \neq j, \\
		&\jProd{c_i}{c_i} = c_i, \, \, \, \, \qquad \qquad  i = 1, \ldots, r, \\
		& c_1 + \cdots + c_r  = e, \, \qquad i = 1, \ldots, r,
		\end{flalign}
and  unique real numbers $\eig _1, \ldots, \eig _r$ satisfying
		\begin{equation}		
		y = \sum _{i=1}^r \eig _i c_i \label{eq:dec}.
		\end{equation}

\end{theorem}
We say that  $c_1, \ldots , c_r$ in Theorem~\ref{theo:spec} form a  \emph{Jordan frame} for $y$, and  
$\lambda _1, \ldots, \lambda _r$ are the \emph{eigenvalues} of $y$.
We 
remark that $r$ only depends on the algebra $\jAlg$. 
Given $y \in \jAlg$, we define its trace by
$$
\tr (y) = \eig _1 + \cdots + \eig _r,
$$
where $\eig _1, \ldots, \eig _r$ are the eigenvalues of 
$y$. As in the case of matrices, it turns out that the trace function 
is  linear. 
It can also be used to define an inner product compatible 
with the Jordan product, and so henceforth we will assume that 
$\inner{x}{y} = \tr(\jProd{x}{y})$. In the case of 
symmetric matrices, $\inner{\cdot}{\cdot}$ turns out to be the Frobenius inner product.

For an element $y \in \jAlg$, we define the \emph{rank} of $y$ as the number of nonzero $\lambda _i$'s that appear in \eqref{eq:dec}. Then, the rank of $\stdCone$ is defined by 
$$
\matRank \stdCone = \max \{ \matRank y \mid y \in \stdCone \} = r = \tr (e).
$$
We will also say that the rank of  $\jAlg$ is $r = \tr (e)$.


For the next theorem, we need the following notation. Given $y\in \jAlg$ and $a \in \Re$, we write 
$$
V(y,a) = \{z \in \jAlg \mid \jProd{y}{z} = az \}.
$$
For any $V,V' \subseteq \jAlg$, we write $\jProd{V}{V'} = \{\jProd{y}{z} \mid y \in V, z \in V'\}$.
\begin{theorem}[Peirce Decomposition -- 1st version, see Proposition IV.1.1 in \cite{FK94}]\label{theo:peirce1}
	Let $c \in \jAlg$ be an idempotent. Then $\jAlg$ is decomposed as the orthogonal direct sum
	$$
	\jAlg = V(c,1) \bigoplus V\left(c,\frac{1}{2}\right) \bigoplus V(c,0).
	$$	
	In addition, $V(c,1)$ and $V(c,0)$ are  Euclidean Jordan algebras satisfying 
	$\jProd{V(c,1)}{V(c,0)} = \{0\}$. Moreover, $\jProd{(V(c,1)+V(c,0))}{V(c,1/2)} \subseteq V(c,1/2)$ and 
	$\jProd{V(c,1/2)}{V(c,1/2)} \subseteq V(c,1) + V(c,0) $.
\end{theorem}

The Peirce decomposition has another version, with  detailed information on the way 
that the algebra is decomposed.

\begin{theorem}[Peirce Decomposition -- 2nd version, see Theorem IV.2.1 in \cite{FK94}]\label{theo:peirce2}
Let $c_1, \ldots, c_r$ be a Jordan frame for $y \in \jAlg$. Then $\jAlg$ is decomposed as the orthogonal sum
$$
\jAlg = \bigoplus _{1\leq i \leq j\leq r} V_{ij},
$$
where
\begin{align*}
V_{ii} &= V(c_i,1) = \{\alpha c_i \mid \alpha \in \Re \}, \\
V_{ij} &= V\left(c_i,\frac{1}{2}\right) \cap V\left(c_j,\frac{1}{2}\right), \qquad \text{  for } i \neq j. 
\end{align*}
Moreover
\begin{enumerate}[label=({\it \roman*})]
	\item the $V_{ii}$'s are subalgebras of $\jAlg$,
	\item the following relations hold:
	\begin{flalign}
	\jProd{V_{ij}}{V_{ij}}  \subseteq V_{ii} + V_{jj} & \qquad \, \, \forall i,j, \label{eq:peirce:1}\\
	\jProd{V_{ij}}{V_{jk}}  \subseteq V_{ik}& \qquad \text{ if } i\neq k,  \label{eq:peirce:2} \\
	\jProd{V_{ij}}{V_{kl}}  = \{0\}& \qquad \text{ if } \{i,j\} \cap \{k,l\} = \emptyset. \label{eq:peirce:3}
	\end{flalign}
\end{enumerate} 

\end{theorem}

The algebra $(\jAlg, \jProd{}{})$ is said to be \emph{simple} if there is no way 
to write $\jAlg = V \oplus W$, where $V$ and $W$ are both nonzero subalgebras of 
$\jAlg$. We will say that $\stdCone$ is \emph{simple} if it is the cone of squares of a 
simple algebra. It turns out that every Euclidean Jordan algebra can be decomposed as  
a direct sum of simple Euclidean Jordan algebras, which then induces a 
decomposition of $\stdCone$ in simple symmetric cones. This means that we can 
write 
\begin{align*}
\jAlg & = \jAlg _1 \oplus \cdots \oplus \jAlg _\ell,\\
\stdCone & = \stdCone _1 \oplus \cdots \oplus \stdCone _\ell,
\end{align*}
where the $\jAlg _i$'s are simple Euclidean Jordan algebras of rank $r_i$ and 
$\stdCone _i$ is the cone of squares of~$\jAlg _i$. Note that orthogonality 
expressed by this decomposition is not only with respect to the inner product $\inner{\cdot}{\cdot}$ 
but also with respect to the Jordan product $\jProd{}{}$.
There is a classification of 
the simple Euclidean Jordan algebras and, up to isomorphism, they fall in four infinite families and a single exceptional case. 

Due to the decomposition results, some articles only deal with simple  Jordan algebras (such as \cite{SS08,FB06}), while others prove results in full generality (such as \cite{Baes04}). The extension from the simple case to the general case is usually straightforward but must be done carefully.

We recall the following properties of $\stdCone$. The results follow 
from various propositions that appear in 
\cite{FK94}, such as Proposition~III.2.2 and Exercise 3 in Chapter III. See also Equation (10) in 
\cite{Sturm2000}. 
\begin{proposition}\label{prop:aux}
Let $y,w \in \jAlg$.
\begin{enumerate}[label=({\it \roman*})]
	\item $y \in \stdCone$ if and only if the eigenvalues of $y$ are nonnegative. \label{paux:1}
	\item $y \in \interior \stdCone$ if and only if the eigenvalues of $y$ are positive. \label{paux:2}
	\item $y \in \interior \stdCone$ if and only if $\inner{y}{\jProd{w}{w}} > 0$ for all 
	nonzero $w \in \jAlg$. \label{paux:3}
	\item Suppose $y,w \in \stdCone$. Then, $\jProd{y}{w} = 0$ if and only if 
	$\inner{y}{w} = 0$. \label{paux:4}
\end{enumerate}
\end{proposition}

%
%
%
%

From item \ref{paux:4} of Proposition~\ref{prop:aux}, we have that if
$c $ and $c'$ are two idempotents belonging to distinct blocks, we also 
have $\jProd{c}{c'} = 0$ in addition to $\inner{c}{c'} = 0$. Since this holds 
for all idempotents, we have $\jProd{\jAlg _i}{\jAlg _j} = 0$, whenever 
$i \neq j$. 

Due to Proposition~\ref{prop:aux}, if $y \in \stdCone$, 
then the eigenvalues of $y$ are nonnegative, and so we can 
define the square root of $y$ as 
\begin{equation*}
\sqrt{y} = \sum _{i = 1}^r \sqrt{\eig _i} c_i,
\end{equation*}
where $\{c_1, \ldots, c_r\}$ is a Jordan frame for $y$.

\subsection{Projection on a symmetric cone}
Denote by $\projC{\stdCone}$ the orthogonal projection on $\stdCone$. 
Given $y \in \jAlg$, $\projC{\stdCone}(y)$ satisfies 
$$
\projC{\stdCone}(y) = \argmin  _{z \in \stdCone} \norm{y-z}.
$$
In analogy to the case of a positive semidefinite cone, to project $x$ on 
$\stdCone$ it is enough to zero the negative eigenvalues of 
$x$. We register this well-known fact as a proposition.

\begin{proposition}\label{prop:proj}
Let $y \in \jAlg$ and consider a Jordan decomposition of $y$
$$
y = \sum _{i = 1}^r \eig _i c_i,
$$
where $\{c_1, \ldots, c_r\}$ is a Jordan frame for $y$. Then, its projection is given by
\begin{equation}
\projC{\stdCone}(y) = \sum _{i = 1}^r \max (\eig _i,0) c_i. \label{eq:proj}
\end{equation}
\end{proposition}
\begin{proof}
Let $z \in \stdCone$. In view of Theorem~\ref{theo:peirce2}, we can write
$$
z = \sum _{i=1}^r v_{ii} + \sum _{1\leq i < j \leq r} v_{ij},
$$
where $v_{ij} \in V_{ij}$ for all $i$ and $j$. As $V_{ii} = \{\alpha c_i \mid \alpha \in \Re\}$, we have 
$$
z = \sum _{i=1}^r \alpha _i c_i + \sum _{1\leq i < j \leq r} v_{ij},
$$
for some constants $\alpha _i \in \Re$.
Recall that the subspaces $V_{ij}$ are all orthogonal among themselves and that $\stdCone$ is 
self-dual. Then, since $c_i \in \stdCone$ and $\inner{z}{c_i} \geq 0$, 
we have  $\alpha _i \geq 0$ for all $i$. 
Furthermore,  we have
$$
\norm{y-z}^2 = \sum _{i=1}^r (\eig_i -\alpha _i )^2 \norm{c_i}^2 + \sum _{1\leq i < j \leq r} \norm{v_{ij}}^2.
$$
Therefore, if we wish to minimize $\norm{y-z}^2$, the best we can do is to set each $v_{ij}$ to zero 
and each $\alpha _i$ to $\max (\eig _i, 0)$. This shows that \eqref{eq:proj} holds.
\end{proof}

In analogy to the symmetric matrices, we will use the following notation:
$$
\projC{\stdCone}(y) = \proj{y}.
$$

The following observation will also be helpful.
\begin{lemma}\label{lemma:moreau}
	Let $\stdCone$ be a symmetric cone and $v \in \jAlg$. Then,
	$$
	v - \proj{v} = - \proj{-v}.
	$$ 
\end{lemma}
\begin{proof}
	The Moreau decomposition (see, e.g., Theorem 3.2.5 in \cite{HL93}) tells us that 
	$v - \proj{v} = P_{\stdCone^\circ}(v)$,
	where $\stdCone^\circ$ is the polar cone of $\stdCone$. As $\stdCone$ is self-dual, we have 
	$\stdCone^\circ = -\stdCone$. Therefore, $P_{\stdCone^\circ}(v) = -P_{\stdCone}(-v) = - \proj{-v}$.
\end{proof}

\subsection{The Karush-Kuhn-Tucker conditions}
First, we define the Lagrangian function $L \colon
\Re^n \times \Re^m \times \jAlg \to \Re$ associated with problem~\eqref{eq:sym} as
\begin{displaymath}
L(x, \mu, \lambda ) \doteq f(x) - \inner{h(x)}{\mult} - \inner{g(x)}{\lambda}.
\end{displaymath}
We say that $(x,\mult,\lambda) \in \Re^n \times \Re^m \times \jAlg$ is a \emph{Karush-Kuhn-Tucker} (KKT) triple of
problem~\eqref{eq:sym} if the following conditions are satisfied:
\begin{align}
\grad f(x) - \jac{h}(x)^* \mult - \jac{g}(x)^* \lambda =  0, \label{eq:kkt_sym.1} \tag{P1.1}\\ 
\lambda \in \stdCone, \label{eq:kkt_sym.2} \tag{P1.2}\\ 
g(x) \in \stdCone, \label{eq:kkt_sym.3} \tag{P1.3}\\ 
\jProd{\lambda}{g(x)}  = 0, \label{eq:kkt_sym.4} \tag{P1.4}\\
h(x)  = 0, \label{eq:kkt_sym.5} \tag{P1.5}
\end{align}
where $\grad f$ is the gradient of $f$,  $\jac{g}$ is the Jacobian of $g$ and $\jac{g}^*$ denotes the adjoint of~$\jac{g}$. Usually, instead of 
\eqref{eq:kkt_sym.4}, we would have $\inner{\lambda}{g(x)} = 0$, but in view of 
item \ref{paux:4} of Proposition~\ref{prop:aux}, they are equivalent.
Note also that \eqref{eq:kkt_sym.1} is equivalent to
$\nabla L_x(x,\mu,\lambda) = 0$, where $\nabla L_x$ denotes the gradient
of $L$ with respect to $x$.

 We also have the following definition.
\begin{definition}\label{def:strict}
	If $(x,\mult, \lambda) \in \Re^n \times \Re^m \times \jAlg$ is a KKT triple of
	\eqref{eq:sym} such that
	\begin{displaymath}
	\matRank g(x) + \matRank \lambda = r,
	\end{displaymath}
	then $(x,\lambda)$ is said to satisfy the \emph{strict
		complementarity} condition.
\end{definition}

As for the equality constrained NLP problem~\eqref{eq:slack}, we observe
that $(x,y,\mult, \lambda) \in \Re^n \times \jAlg\times \Re^m \times \jAlg$ is a KKT quadruple
if the conditions below are satisfied:
\begin{align*}
\grad_{(x,y)} \mathcal{L}(x,y,\mult, \lambda) & = 0, \\
h(x) & = 0,\\
g(x) - \jProd{y}{y} & = 0,
\end{align*}
where $\mathcal{L} \colon \Re^n \times \jAlg \times \Re^m \times \jAlg \to \Re$ is
the Lagrangian function associated with~\eqref{eq:slack}, which is given by
\begin{displaymath}
\mathcal{L}(x,y,\mult,\lambda) \doteq f(x) - \inner{h(x)}{\mult}-\inner{g(x) - y \circ
	y}{\lambda}
\end{displaymath}
and $\grad_{(x,y)}  \mathcal{L}$ denotes the gradient of $\mathcal{L}$ with respect to $(x,y)$.

We can then write the KKT conditions for~\eqref{eq:slack} as
\begin{align}
  \grad f(x) -  \jac{h}(x)^* \mult -  \jac{g}(x)^* \lambda = 0, \label{eq:kkt_slack.1} 
  \tag{P2.1}\\
  \jProd{\lambda}{y} = 0, \label{eq:kkt_slack.2} \tag{P2.2}\\
  g(x) - \jProd{y}{y} = 0, \label{eq:kkt_slack.3} \tag{P2.3}\\
    h(x)  = 0. \label{eq:kkt_slack.4} \tag{P2.4}
\end{align}

Writing the conditions for~\eqref{eq:sym} and \eqref{eq:slack}, we see that although 
they are equivalent problems, the KKT conditions are slightly different. In fact, for 
\eqref{eq:slack}, it is not required that $\lambda$ belongs to $\stdCone$ and this accounts for 
a great part of the difference between both formulations.

For \eqref{eq:sym}, we say that the \emph{Mangasarian-Fromovitz constraint 
qualification} (MFCQ) holds at a point $x$ if $\jac{h}(x)$ is surjective and there exists some $d \in \Re^n$ such that 
\begin{align*}
\jac{h}(x)d &= 0,\\
g(x) + \jac{g}(x)d &\in \interior \stdCone,
\end{align*}
where $\interior \stdCone $ denotes the interior of $\stdCone$. See, for instance, 
Equation~(2.190) in \cite{BS00}.
If $x$ is a local minimum for \eqref{eq:sym}, MFCQ ensures the existence of 
a pair of Lagrange multipliers $(\mu,\lambda)$ and that the set of multipliers is bounded.

We also can define a nondegeneracy condition  as follows.
\begin{definition}\label{def:nondeg}
  Suppose that $x \in \stdCone$ is such that
\begin{align*}
\Re^m \times \jAlg = \mathrm{Im}\,\begin{pmatrix}
\jac{h}(x)\\
\jac{g}(x)
\end{pmatrix} + \{0\}^m\times \lineality \tanCone{g(x)}{\stdCone}, 
\end{align*}
where $\mathrm{Im}\,\begin{psmallmatrix}
	\jac{h}(x)\\
	\jac{g}(x)
\end{psmallmatrix}$ denotes the image of the linear map
that takes $d \in \Re^n$ to $\begin{psmallmatrix}
	\jac{h}(x)d\\
	\jac{g}(x)d
	\end{psmallmatrix} \in \Re^m \times \jAlg $,
	$\{0\}^m$ denotes the zero subspace in $\Re^m$, $\tanCone{g(x)}{\stdCone}$ denotes
  the tangent cone of $\stdCone$ at $g(x)$, and  $\lineality \tanCone{g(x)}{\stdCone}$ is the lineality space of $\tanCone{g(x)}{\stdCone}$, i.e., $\lineality \tanCone{g(x)}{\stdCone} = \tanCone{g(x)}{\stdCone} \cap -\tanCone{g(x)}{\stdCone}$.
  Then, $x$ is said to be \emph{nondegenerate}.
\end{definition}

This definition is related to the transversality condition introduced by Shapiro in \cite{shapiro97}, see Definition 4 therein. 
Furthermore, Definition~\ref{def:nondeg} is a special case of a more general definition of nondegeneracy discussed in the work by Bonnans and Shapiro, see Section 4.6.1 of \cite{BS00}.

For \eqref{eq:slack}, we say that the \emph{linear independence constraint qualification} (LICQ) is satisfied at a point $(x,y)$ if the gradients of the constraints are linearly independent.

\subsection{Second order conditions for \eqref{eq:slack}}
For~\eqref{eq:slack}, we say that the second order sufficient condition (SOSC-NLP)  holds if 
\begin{displaymath}
    \inner{\grad _{(x,y)} ^2 \mathcal{L}(x,y,\mult,\lambda)(v,w)}{(v,w)} > 0,
\end{displaymath}
for every nonzero $(v,w) \in \Re^n \times \jAlg$ such that $\jac{g}(x)v - 2\jProd{y}{w} = 0$ and 
$\jac{h}(x)v = 0$, where $\grad _{(x,y)} ^2 \mathcal{L}$ denotes the Hessian of
$\mathcal{L}$ with respect to $(x,y)$. See 
\cite[Section 3.3]{Ber99} or \cite[Theorem 12.6]{NW99}. 
We can also present the SOSC-NLP in terms 
of the Lagrangian of \eqref{eq:sym}.

\begin{proposition}\label{prop:sosc_slack}
  Let $(x,y,\mult,\lambda) \in \Re^n \times \jAlg \times \Re^m \times \jAlg$ be a KKT quadruple of problem~\eqref{eq:slack}. The
  SOSC-NLP holds if
  \begin{equation}
    \inner{\grad _x ^2 L(x,\mult,\lambda)v}{v} + 2\inner{\jProd{w}{w}}{\lambda} > 0  \label{eq:sosc_nlp.1}
  \end{equation}
  for every nonzero $(v,w) \in \Re^n \times \jAlg$ such that $\jac{g}(x)v - 2\jProd{y}{w} = 0$ and 
  $\jac{h}(x)v = 0$.
\end{proposition}
\begin{proof}
Note that we have
\begin{align*}
\grad _{(x,y)} ^2 \mathcal{L}(x,y,\mult,\lambda) & = \grad _{(x,y)} ^2 [L(x,\mult,\lambda) + \inner{\jProd{y}{y}}{\lambda} ]. 
 \end{align*}
 Therefore, 
$$
\inner{\grad _{(x,y)} ^2 \mathcal{L}(x,y,\mult,\lambda)(v,w)}{(v,w)} = \inner{\grad _x ^2 L(x,\mult,\lambda)v}{v} + \inner{\grad _y^2 \inner{\jProd{y}{y}}{\lambda} w }{w}.$$
Due to the fact that the underlying algebra is Euclidean, we 
have  $ \grad _y \inner{\jProd{y}{y}}{\lambda} = 2 \jProd{y}{\lambda} $ and 
$\grad _y^2 \inner{\jProd{y}{y}}{\lambda} = 2L_\lambda$, where 
we recall that $L_\lambda$ is the linear operator satisfying $L_{\lambda}(z) = \jProd{\lambda }{z}$ for every $z$. We then conclude 
that 
$$\inner{\grad _y^2 \inner{\jProd{y}{y}}{\lambda} w }{w} = \inner{w}{2L_\lambda(w)} = 2\inner{\jProd{w}{w}}{\lambda},$$
which implies that \eqref{eq:sosc_nlp.1} holds.
\end{proof}

Similarly, we have the following second order  necessary condition (SONC). Note that we require the 
LICQ to hold.

\begin{proposition}\label{prop:sonc_slack}
Let $(x,y)$ be a local minimum for \eqref{eq:slack} 
and  $(x,y,\mult,\lambda) \in \Re^n \times \jAlg \times \Re^m \times \jAlg$ be a KKT quadruple such 
that LICQ holds. Then, the following SONC-NLP holds: 
\begin{equation}
    \inner{\grad _x ^2 L(x,\mult,\lambda)v}{v} + 2\inner{\jProd{w}{w}}{\lambda} \geq 0 \label{eq:sonc_nlp}
  \end{equation}
  for every $(v,w) \in \Re^n \times \jAlg$ such that $\jac{g}(x)v - 2\jProd{y}{w} = 0$ and $\jac{h}(x)v = 0$.
\end{proposition}
\begin{proof}
  See \cite[Theorem 12.5]{NW99} for the basic format of the second order necessary condition
  for NLPs.   In order to express the condition in terms of the Lagrangian of 
  \eqref{eq:sym}, we proceed as in the proof of Proposition~\ref{prop:sosc_slack}.
\end{proof}

Under the assumption that $\jAlg$ is decomposed as $\jAlg = \jAlg _1 \oplus \cdots \oplus \jAlg _\ell $, the term  
$2\inner{\jProd{w}{w}}{\lambda} $ can be written as 
$$
\sum _{i = 1}^\ell 2\inner{\jProd{w_i}{w_i}}{\lambda _i},
$$
which brings Propositions \ref{prop:sosc_slack} and \ref{prop:sonc_slack} closer to the format described, for instance, in Lemma 3.2 of \cite{FF16}.

Furthermore, we observe that in Propositions \ref{prop:sosc_slack} and \ref{prop:sonc_slack} an extra term appears together 
with the Lagrangian of \eqref{eq:sym}. This term is connected with the so-called \emph{sigma-term} that appears in second order optimality conditions for optimization problems over general closed convex cones and plays an important role in the construction of no-gap optimality conditions. Classical references for second order optimality conditions for general cones include the works of Kawasaki \cite{Kawasaki88}, Cominetti \cite{Cominetti1990}, Bonnans, Cominetti and Shapiro \cite{BonnansCS}, and the book by Bonnans and Shapiro \cite{BS00}. 
In particular, we refer to sections 3.2, 3.3, 5.2 and 5.3 in \cite{BonnansCS}.

\section{A criterion for membership in $\stdCone$}\label{sec:crit}

We need the following auxiliary result.
\begin{proposition}\label{prop:trick}
Let $(\jAlg,\circ )$ be a \emph{simple} Euclidean Jordan algebra and let 
$c_i$ and $c_j$ be two orthogonal primitive idempotents. Then 
$$
V\left(c_i,\frac{1}{2}\right) \cap V\left(c_j,\frac{1}{2}\right) \neq \{0\}.
$$
\end{proposition}
\begin{proof}
See Corollary IV.2.4 in \cite{FK94}.
%
%
\end{proof}

Since $\stdCone$ is self-dual, we have that $\lambda \in \stdCone$ if and only 
if $\inner{\lambda}{\jProd{w}{w}} \geq 0$ for all $w \in \jAlg$. Lemma~\ref{lemma:cone} refines this criterion for the 
case where $\stdCone$ is \emph{simple}.

\begin{lemma}\label{lemma:cone}
	Let $(\jAlg,\circ )$ be a \emph{simple} Euclidean Jordan algebra of rank $r$ and $\lambda \in \jAlg$.
	The following statements are equivalent:
	\begin{enumerate}[label=({\it \roman*})]
	  \item $\lambda \in \stdCone$.
	  \item There exists $y \in \jAlg$ such that $\jProd{y}{\lambda} = 0$ and
	  \begin{equation}\label{eq_lemma}
	  	\inner{\jProd{w}{w}}{ \lambda } > 0, 
	  	\end{equation}
	  	for every $w \in \jAlg$ satisfying $\jProd{y}{w} = 0$ and $w \neq 0$.
	\end{enumerate}
	
Moreover, any $y$ satisfying 	$(ii)$ is such that 
\begin{enumerate}[label=({\it \alph*})]
	  \item $\matRank y = r - \matRank \lambda$, i.e., $y$ and $\lambda$ satisfy strict complementarity,
	  \item if $\eig $ and $\eig '$ are non-zero eigenvalues of $y$, then 
	  $\eig + \eig ' \neq 0$.
	\end{enumerate}	 
\end{lemma}
\begin{proof}
\fbox{$ (i) \Rightarrow (ii)$} If $\lambda \in \stdCone$, we write its spectral decomposition as
\begin{equation*}
\lambda = \sum _{i=1}^r \eig _i c_i,
\end{equation*}
where we assume that only the first $\matRank \lambda$ eigenvalues are positive and 
the others are zero.
If $\matRank \lambda = r$, we take $y = 0$. 
Otherwise, take $$y = \sum _{i = \matRank \lambda +1}^r c_i.$$ 
Note that $y$ is an idempotent and 
that $\lambda$ lies in the relative interior of the cone of squares of the Jordan algebra 
$V(y,0)$. Hence, the condition  \eqref{eq_lemma} is satisfied.

\fbox{$ (ii) \Rightarrow (i)$, together with $(a)$ and $(b)$}	
We write
\begin{equation*}
y = \sum _{i=1}^{r}\eig _i c_i = \sum _{i=1}^{ \matRank y} \eig _i c_i,
\end{equation*}
where $\{c_1,\ldots, c_r\}$ is a Jordan frame, and we assume that the first $\matRank y$ eigenvalues of $y$ are nonzero and 
the others are zero. Then, following Theorem~\ref{theo:peirce2}, we write
\begin{equation*}
\lambda = \sum _{i\leq j} \lambda _{ij} = \sum _{i=1}^r \lambda _{ii} + \sum _{i< j} \lambda _{ij},
\end{equation*}
where $\lambda _{ij} \in V_{ij}$. Using the operation rules
in Theorem~\ref{theo:peirce2}, we get
\begin{equation}
\jProd{c_k}{\lambda _{ij}} = 
  \begin{cases}
   \lambda _{ij},      & \quad \text{ if }  i=j=k\\
  0,  & \quad \text{ if }  \{k\} \cap \{i,j\} = \emptyset \\
   \frac{\lambda _{ij}}{2},  & \quad \text{ if }  i< j, \quad \{k\} \cap \{i,j\} = \{k\}. \\
  \end{cases} \nonumber
\end{equation}
Therefore,
\begin{equation}
\jProd{y}{\lambda} =  \sum _{i = 1}^{\matRank y} \sigma _i \lambda_{ii} + \sum _{1\leq i < j \leq \matRank y} \left(\frac{\sigma_i + \sigma _j}{2}\right) \lambda_{ij} + \sum _{1\leq i \leq \matRank y< j} \frac{\sigma _i }{2}\lambda_{ij}. \label{eq:y_lb}
\end{equation}
By hypothesis, we have $\jProd{y}{\lambda} = 0$. Since the $V_{ij}$'s are mutually orthogonal subspaces, we 
conclude that all terms inside the summations in \eqref{eq:y_lb} must be zero. 
In particular, we have $ \sigma _i \lambda _{ii} = 0$, for every $i \leq \matRank y$. 
As the $\sigma _i$'s are 
nonzero for those indexes, we have $\lambda _{ii} = 0$ so that
\begin{equation}
\lambda =  \sum _{\matRank y < i \leq r} \lambda _{ii} + \sum _{i< j} \lambda _{ij} \label{eq:lb_dec},
\end{equation}

We now show that 
$(b)$ holds. Suppose for the sake of contradiction that 
$\sigma _i + \sigma _j = 0$ for some $i < j \leq \matRank y$. 
By Proposition~\ref{prop:trick}, there is a nonzero 
$w \in V(c_i,1/2)\cap V(c_j, 1/2)$. Since $\jProd{w}{c_k} = 0$ for $k \neq i$, $k \neq j$,
we have 
$$\jProd{y}{w} = (\jProd{\eig _i c_i}{w}) + (\jProd{\eig _j c_j}{w})  = \left(\frac{\eig _i + \eig _j}{2}\right)w = 0.
$$
Moreover, $\jProd{w}{w} \in V(c_i,1) + V(c_j,1)$, due to \eqref{eq:peirce:1}.  By \eqref{eq:lb_dec} and the orthogonality among the $V_{ij}$, 
$$   \inner{\jProd{w}{w}}{\lambda} =  0,  $$
since $\lambda$ has no component in neither $V(c_i,1) = V_{ii}$ nor  $V(c_j,1) = V_{jj}$.
This contradicts \eqref{eq_lemma}, and so it must be the case that $(b)$ holds. 

Let $c = c_1 +\cdots+ c_{\matRank y}$. We will now show that $V(c,0) = V(y,0)$.
Since $\jAlg = V(c,0) \oplus V(c,1/2) \oplus V(c,1)$ and $y \in V(c,1)$, we 
have $V(c,0) \subseteq V(y,0)$, since $\jProd{V(c,0)}{V(c,1)} = \{0\}$. 

The 
next step is to prove that $V(y,0) \subseteq V(c,0)$. Suppose that $\jProd{y}{w} = 0$ and 
write $w = \sum _{i\leq j} w_{ij}$ with $w_{ij} \in V_{ij}$ as in Theorem~\ref{theo:peirce2}.
As in \eqref{eq:y_lb}, we have
\begin{equation*}
\jProd{y}{w} = \sum _{i = 1}^{\matRank y} \sigma _i w_{ii} + \sum _{1\leq i < j \leq \matRank y} \left(\frac{\sigma_i + \sigma _j}{2}\right) w_{ij} + \sum _{1\leq i \leq \matRank y< j} \frac{\sigma _i }{2}w_{ij} = 0. 
\end{equation*}
Because $\sigma _i,\sigma _j$ and $\sigma _i + \sigma _j$ are all nonzero when 
$i\leq  j \leq \matRank$, it follows that 
$w_{ij} = 0$ for $i\leq  j \leq \matRank y$ and for $i\leq  \matRank y < j$, since $w_{ij}$ all lie in 
mutually orthogonal subspaces. Therefore, 
$w = \sum _{\matRank y < i \leq j} w_{ij}$ and Theorem~\ref{theo:peirce2} implies 
that $\jProd{c}{w} = 0$, which shows that $V(y,0) \subseteq V(c,0)$.

We now know that $V(c,0) = V(y,0)$. By Theorem~\ref{theo:peirce1}, $V(c,0)$ is an Euclidean Jordan 
algebra and its  rank is $r - \matRank y$, since $c_{\matRank y + 1} + \cdots + c_r$ is 
the identity in $V(c,0)$.
Then, the condition  \eqref{eq_lemma} means 
that $\inner{z}{\lambda} > 0$ for all $z$ in the cone of squares of the
algebra $V(c,0)$. By item \ref{paux:3} of Proposition~\ref{prop:aux}, this means 
that $\lambda$ belongs to the (relative) interior of $\stdCone ' = \{\jProd{w}{w} \mid  w \in V(c,0)\}$.
This shows both that $\lambda  \in \stdCone$ and that $\matRank \lambda =  \matRank \stdCone ' = r - \matRank y$. These are items $(i)$ and $(a)$.
\end{proof}
Lemma~\ref{lemma:cone} does not apply directly when $\stdCone$ is not simple because the 
complementarity displayed in item $(b)$ only works ``inside the same blocks''. That is essentially 
the only aspect we need to account for. As a reminder, 
we have
\begin{align*}
\jAlg & = \jAlg _1 \oplus \cdots \oplus \jAlg _\ell,\\
\stdCone & = \stdCone _1 \oplus \cdots \oplus \stdCone _\ell,
\end{align*}
where the $\jAlg _i$ are simple Euclidean Jordan algebras of rank $r_i$ and 
$\stdCone _i$ is the cone of squares of $\jAlg _i$. The rank of 
$\jAlg$ is $r = r_1 + \cdots + r_\ell$. 

\begin{theorem}\label{theo:strict}
Let $(\jAlg, \jProd{}{})$ be an Euclidean Jordan algebra of	rank $r$ and $\lambda \in \jAlg$.
The following statements are equivalent:
\begin{enumerate}[label=({\it \roman*})]
	\item $\lambda \in \stdCone$.
	\item There exists $y \in \jAlg$ such that $\jProd{y}{\lambda} = 0$ and
	\begin{equation}\label{eq:theo}
	\inner{\jProd{w}{w}}{ \lambda } > 0, 
	\end{equation}
	for every $w \in \jAlg$ satisfying $\jProd{y}{w} = 0$ and $w \neq 0$.
\end{enumerate}
Moreover, any $y$ satisfying 	$(ii)$ is such that 
\begin{enumerate}[label=({\it \alph*})]
	\item $\matRank y = r - \matRank \lambda$, i.e., $y$ and $\lambda$ satisfy strict complementarity,
	\item if $\eig $ and $\eig '$ are non-zero eigenvalues of $y$ belonging to the same block, then 
	$\eig + \eig ' \neq 0$.
\end{enumerate}	 
\end{theorem}
\begin{proof}
\fbox{$ (i) \Rightarrow (ii)$} Write $\lambda = \lambda _1 + \cdots + \lambda _\ell$, according to block division in $\jAlg$.
Then apply Lemma~\ref{lemma:cone} to each $\lambda _i$ to obtain $y_i$, and let $y = y_1 + \cdots + y_\ell$.

\fbox{$ (ii) \Rightarrow (i)$}  Write $\lambda = \lambda _1 + \cdots + \lambda _\ell$ and 
$y = y_1 + \cdots + y_\ell$. Then, the inequality  \eqref{eq:theo} implies that for 
every $i$,
\begin{equation}
	\inner{\jProd{w_i}{w _i}}{ \lambda _i } > 0, 
\end{equation}
for all nonzero $w_i$ with $\jProd{w_i}{y_i} = 0$. Therefore, Lemma~\ref{lemma:cone} applies to each 
$y_i$, thus concluding the proof.\end{proof}

Theorem \ref{theo:strict} extends Lemma~1 in \cite{LFF16} for all symmetric cones. 
For a product of second order cones, no similar result is explicitly stated in \cite{FF16}, but it can be derived from the proof of Proposition 3.2 therein. 

\section{Comparison of KKT points for \eqref{eq:sym} and \eqref{eq:slack}}\label{sec:kkt}
Although \eqref{eq:sym} and \eqref{eq:slack} share the same global minima, their KKT points are not necessarily the same. However, if $(x,\mu,\lambda)$ is a KKT triple for \eqref{eq:sym}, it is 
easy to construct a KKT quadruple for \eqref{eq:sym} according to the next proposition.

\begin{proposition}
Let $(x,\mult,\lambda) \in  \Re^n \times \Re^m \times \jAlg$ be a KKT triple for \eqref{eq:sym}, then $(x,\sqrt{g(x)},\mult,\lambda)$ is a 
KKT quadruple for \eqref{eq:slack}.
\end{proposition}
\begin{proof}
The quadruple $(x,\sqrt{g(x)},\mult,\lambda)$ satisfies \eqref{eq:kkt_slack.1}, \eqref{eq:kkt_slack.3} and \eqref{eq:kkt_slack.4}.
We will now check that \eqref{eq:kkt_slack.2} is also satisfied.
We can 
write 
$$
\lambda = \sum _{i = 1}^{\matRank \lambda} \eig _i c_i,
$$
where  $\{c_1,\ldots, c_r\}$ is a Jordan frame for $\lambda$ such that 
$\eig _i > 0$ for $i = 1,\ldots, \matRank \lambda$.
By item \ref{paux:4} of Proposition~\ref{prop:aux} and \eqref{eq:kkt_sym.4}, we 
have $\inner{\lambda}{g(x)} = 0$, which implies that 
$\inner{c_i}{g(x)} = 0$ for $i = 1,\ldots, \matRank \lambda$. 
Again, by item \ref{paux:4} of Proposition~\ref{prop:aux}, we obtain 
that $\jProd{c_i}{g(x)} = 0$ for $i = 1,\ldots, \matRank \lambda$. 

Let $c = c_1+ \cdots + c_{\matRank \lambda}$. Using Theorem~\ref{theo:peirce1}, we 
write
$$
\jAlg = V(c,1)\bigoplus V(c,1/2) \bigoplus V(c,0).
$$ 
We then have $\lambda \in V(c,1)$ and $g(x) \in V(c,0)$. Because $V(c,0)$ is 
also an Euclidean Jordan algebra, we have $\sqrt{g(x)} \in V(c,0)$. Finally, 
since $\jProd{V(c,1)}{V(c,0)} = \{0\}$, we readily obtain $\jProd{\lambda}{\sqrt{g(x)}} = 0$, 
which is \eqref{eq:kkt_slack.2}.
\end{proof}

It is not true in general that if $(x,y,\mult,\lambda)$ is a KKT quadruple for 
\eqref{eq:slack}, then $(x,\mu,\lambda)$ is a KKT triple for \eqref{eq:sym}. 
Nevertheless, the only obstacle is  that $\lambda$ might fail to 
belong to $\stdCone$.

\begin{proposition}\label{prop:kkt}
Let $(x,y,\mult,\lambda) \in  \Re^n \times \jAlg \times \Re^m \times \jAlg$ be a KKT quadruple for \eqref{eq:slack} such that 
$\lambda \in \stdCone$. Then, $(x,\mult,\lambda)$ is a KKT triple for 
\eqref{eq:sym}.
\end{proposition}
\begin{proof}
Under the current hypothesis, \eqref{eq:kkt_sym.1}, \eqref{eq:kkt_sym.2}, \eqref{eq:kkt_sym.3} and 
\eqref{eq:kkt_sym.5}
are satisfied. Due to \eqref{eq:kkt_slack.3}, we have 
\begin{align*}
0  = \inner{y}{\jProd{y}{\lambda}} = \inner{\jProd{y}{y}}{\lambda} = \inner{g(x)}{\lambda},
\end{align*}
where the second equality follows from Property 3 of the 
$\jAlg{}{}$ operator.
Therefore, by item \ref{paux:4} of Proposition~\ref{prop:aux}, we 
obtain $\jProd{g(x)}{\lambda} = 0$, which is \eqref{eq:kkt_sym.4}.
\end{proof}

We then have the following immediate consequence.
\begin{proposition}\label{prop:kkt2}
Let $(x,y,\mult,\lambda) \in \Re^n \times \jAlg \times \Re^m \times \jAlg$ be a KKT quadruple for \eqref{eq:slack}.
\begin{enumerate}[label=({\it \roman*})]
\item If $y$ and $\lambda$ satisfy the assumptions of item $(ii)$ of Theorem~\ref{theo:strict}, 
then $(x,\mult, \lambda)$ is a 
KKT triple for \eqref{eq:sym} satisfying strict complementarity. 
\item If SOSC-NLP holds at $(x,y,\mult,\lambda)$, then $(x,\mult, \lambda)$ is a 
KKT triple for \eqref{eq:sym} satisfying strict complementarity.
\end{enumerate}
\end{proposition}
\begin{proof}
\begin{enumerate}[label=({\it \roman*})]
\item By Theorem~\ref{theo:strict}, $\lambda \in \stdCone$. Therefore, 
by Proposition~\ref{prop:kkt}, $(x,\mult, \lambda)$ is a 
KKT triple for \eqref{eq:sym}. Moreover, due to 
item $(a)$ of Theorem~\ref{theo:strict}, we have
$\matRank y + \matRank \lambda = r$. As $g(x) = \jProd{y}{y}$, we 
have $\matRank g(x) + \matRank \lambda = r$ as well.

\item If SOSC-NLP holds at $(x,y,\mult,\lambda)$, then taking 
$v = 0$ in \eqref{eq:sosc_nlp.1} in Proposition~\ref{prop:sosc_slack}, we obtain that $y$ and $\lambda$ satisfy the assumptions of item $(ii)$ of Theorem~\ref{theo:strict}. Therefore, the result follows 
from the previous item.

\end{enumerate}
\end{proof}
Note that Propositions \ref{prop:kkt} and \ref{prop:kkt2} extend previous results obtained in Section 3 of \cite{FF16} for NSOCPs and in Section 3 of \cite{LFF16} for NSDPs. 
In comparison, in this work,  besides the fact that we are focused on a more general family of cones, we are also considering equality constraints.  

\section{Comparison of constraint qualifications}\label{sec:cons}
In order to understand the differences between constraint qualifications for 
\eqref{eq:sym} and \eqref{eq:slack}, we first have to understand 
the shape of the tangent cones of $\stdCone$. A description can be 
found in Section 2.3 in the work by Kong, Tun{\c{c}}el and Xiu \cite{Kong2011}. 
Apart from that, we can also the use the relations described by Pataki in Lemma~2.7 of \cite{pataki_handbook}.
For the sake of self-containment, we will give an account of the theory. 
In what follows, if $C\subseteq \jAlg$, we define $C^\perp \coloneqq \{z \in \jAlg \mid \inner{z}{y} = 0, \forall y \in C \}$.

Let $z \in \stdCone$, where the rank of $\stdCone$ is $r$. We will now proceed to describe the shape of $\tanCone{z}{\stdCone}$ and 
$\lineality \tanCone{z}{\stdCone}$. First, denote by $\minFacePoint{z}{\stdCone}$ the \emph{minimal face} of $\stdCone$ 
which contains $z$. Denote by $\minFacePoint{z}{\stdCone}^\Delta$ the \emph{conjugated face} of  $\minFacePoint{z}{\stdCone}$, which 
is defined as $\stdCone^* \cap \minFacePoint{z}{\stdCone}^\perp$. Since, $\stdCone$ is self-dual, we have 
$\minFacePoint{z}{\stdCone}^\Delta = \stdCone \cap \minFacePoint{z}{\stdCone}^\perp$.
Now, the discussion in Section 2 and Lemma~2.7 of \cite{pataki_handbook} shows that 
\begin{align*}
\minFacePoint{z}{\stdCone}^\Delta & =  \stdCone \cap \{z\}^\perp,\\
\tanCone{z}{\stdCone} & = \minFacePoint{z}{\stdCone}^{\Delta*},\\
\lineality \tanCone{z}{\stdCone} & = \minFacePoint{z}{\stdCone}^{\Delta\perp}.
\end{align*}

Our next task is to describe $\minFacePoint{z}{\stdCone}$. Let $\{c_1,\ldots, c_r\}$ be a Jordan 
frame for $z$ and write the spectral decomposition of 
$z$ as
$$
z = \sum _{i = 1}^{\matRank z} \eig _i c_i,
$$
where $\eig _1, \ldots, \eig _{\matRank z}$ are positive. Now, define $c = c_1+\cdots + c_{\matRank z}$.
Then, $c$ is an idempotent and Theorem~\ref{theo:peirce1} implies that
$$
\jAlg = V(c,1) \bigoplus V\left(c,\frac{1}{2}\right) \bigoplus V(c,0).
$$
A result by Faybusovich  (Theorem~2 in \cite{FB06}) implies that 
$\minFacePoint{z}{\stdCone} $ is the cone of squares in $V(c,1)$, that is,
$$
\minFacePoint{z}{\stdCone} = \{\jProd{y}{y} \mid  y \in V(c,1)\}.
$$
Then, we can see that $\minFacePoint{z}{\stdCone}^\Delta$ is precisely the cone of 
squares of $V(c,0)$. We remark this  fact as a proposition.
\begin{proposition}
$\minFacePoint{z}{\stdCone} ^\Delta = \{\jProd{y}{y} \mid y \in V(c,0) \}$. 
\end{proposition}
\begin{proof}
We first show that 	$\minFacePoint{z}{\stdCone} ^\Delta \subseteq \{\jProd{y}{y} \mid y \in V(c,0) \}$.
If $w \in \stdCone$ and $\inner{w}{z} = 0$, then we must have $\inner{c_i}{w} = 0$, for every 
$i \in \{1, \ldots, \matRank z \}$. Then Lemma~\ref{lemma:cone} implies that 
$\jProd{c_i}{w} = 0$ for those $i$. This shows that $\jProd{c}{w} = 0$, so that 
$w \in V(c,0)$. As $w \in \stdCone$ and $V(c,0)$ is an Euclidean Jordan algebra, we have $w = \jProd{y}{y}$ for some $y \in V(c,0)$.

Now, let $w = \jProd{y}{y}$ with $y \in V(c,0)$. As $z \in V(c,1)$ and $w \in V(c,0)$ we have 
$\inner{w}{z} = 0$, so that $w \in \minFacePoint{z}{\stdCone} ^\Delta$.
\end{proof}
If we restrict ourselves to $V(c,0)$, then $\minFacePoint{z}{\stdCone} ^\Delta $ is a genuine symmetric 
cone, since it is a cone of squares induced by an Euclidean Jordan algebra. In particular, $\minFacePoint{z}{\stdCone} ^\Delta$ is self-dual in the sense that 
$\minFacePoint{z}{\stdCone} ^\Delta = \{w \in V(c,0) \mid \inner{w}{v} \geq 0, \forall v \in \minFacePoint{z}{\stdCone} ^\Delta \}$.
Following the Peirce decomposition, we conclude that 
\begin{align}
\tanCone{z}{\stdCone}  = \minFacePoint{z}{\stdCone}^{\Delta*} & = V(c,1) \bigoplus V\left(c,\frac{1}{2}\right) \bigoplus \minFacePoint{z}{\stdCone}^{\Delta}, \label{eq:tan_cone} \\
\lineality \tanCone{z}{\stdCone}  = \minFacePoint{z}{\stdCone}^{\Delta\perp} & = V(c,1) \bigoplus V\left(c,\frac{1}{2}\right) \bigoplus \{0\}, \label{eq:lin_cone}
\end{align}
where we recall that $\lineality \tanCone{z}{\stdCone}$ denotes the largest subspace contained in the cone $\tanCone{z}{\stdCone}$.

We are now prepared to discuss the difference between constraint qualifications for \eqref{eq:sym} and \eqref{eq:slack}.
This discussion is analogous to the one in Section \ref{def:nondeg} of \cite{LFF16}. 
We remark that a similar discussion for the special case of nonlinear programming appears in Section 3 of the work by  Jongen and Stein \cite{JS03}.
First, we recall that nondegeneracy for \eqref{eq:sym} at a point $x$ is the 
same as saying that the following condition holds: 
\begin{align}
w \in (\lineality \tanCone{g(x)}{\stdCone})^\perp,  \jac{g}(x)^*w + \jac{h}(x)^*v = 0&  \quad \implies \quad w = 0, v = 0. \tag{Nondegeneracy}\label{eq_nondegeneracy2}
\end{align}
On the other hand, LICQ holds for \eqref{eq:slack} at a point $(x,y)$ if 
the following  condition holds:
\begin{align}
\jProd{w}{y} = 0,\jac{g}(x)^*w + \jac{h}(x)^*v = 0 & \quad \implies \quad w = 0, v= 0. \tag{LICQ}\label{eq_licq}
\end{align}

We need the following auxiliary result. 
\begin{proposition}\label{prop:nondeg_licq}
	Let $z = \jProd{y}{y}$. Then $(\lineality \tanCone{z}{\stdCone})^\perp \subseteq \ker L_y$, where $\ker L_y$ is the kernel of $L_y$. 
	If $y \in \stdCone$, then $\ker L_y \subseteq (\lineality \tanCone{z}{\stdCone})^\perp$ as 
	well.
\end{proposition}
\begin{proof}
Using \eqref{eq:lin_cone}, we have 
$$
(\lineality \tanCone{z}{\stdCone})^\perp =  \{0\} \bigoplus \{0\} \bigoplus V(c,0),
$$
where we assume that  $z =  \sum _{i = 1}^{\matRank z} \eig _i c_i$ with $\eig _i > 0$ for $i \leq \matRank z$ and 
$c = c_1 + \cdots + c_{\matRank z}$.
Let $w \in V(c,0)$. Recall that $y$ and $z$ share a Jordan frame, so we may assume that  $y \in V(c,1)$. Since 
$ \jProd{V(c,0)}{V(c,1)} = \{0\}$, we see that $\jProd{y}{w} = 0$, that is, 
$w \in \ker L_y$. This shows that $(\lineality \tanCone{z}{\stdCone})^\perp \subseteq \ker L_y$.

Now, suppose that $y \in \stdCone$, $w \in \ker L_y$. Since $y$ is the square root of 
$z$ that belongs to $\stdCone$, we may assume that $y = \sum _i \sqrt{\eig _i} c_i$.
Then, we decompose $w$ as $w =\sum _{i\leq j} w_{ij}$, as 
in Theorem~\ref{theo:peirce2}, with $w_{ij} \in V_{ij}$. Then, as in \eqref{eq:y_lb}, we have
\begin{align*}
\jProd{y}{w} = \sum _{i = 1}^{\matRank y} \sqrt{\sigma _i} w_{ii} + \sum _{1\leq i < j \leq \matRank y} \left(\frac{\sqrt{\sigma_i} + \sqrt{\sigma _j}}{2}\right) w_{ij} + \sum _{1\leq i \leq \matRank y< j} \frac{\sqrt{\sigma _i} w_{ij}}{2} = 0. 
\end{align*}
The condition $\jProd{y}{w} = 0$, the fact that the $\sqrt{\eig _i}$ are positive, and the orthogonality among the $w_{ij}$ imply that 
$w = \sum _{\matRank y< i \leq j} w_{ij}$, so that $\jProd{w}{c} = 0$. Hence $w \in (\lineality \tanCone{z}{\stdCone})^\perp$.
\end{proof}

\begin{corollary}\label{col:licq_nondeg}
	If $(x,y) \in \Re^n \times \stdCone$ satisfies LICQ for  problem \eqref{eq:slack}, then nondegeneracy is 
	satisfied at $x$ for \eqref{eq:sym}. On the other hand, if $x$ satisfies nondegeneracy
	and if $y = \sqrt{g(x)}$, then $(x,y)$ satisfies LICQ for \eqref{eq:slack}.
\end{corollary}

\begin{proof}
Follows from combining Proposition~\ref{prop:nondeg_licq} with \eqref{eq_licq} and 
\eqref{eq_nondegeneracy2}.
\end{proof}

\section{Second order conditions for \eqref{eq:sym}}\label{sec:soc}
Using the connection between \eqref{eq:sym} and \eqref{eq:slack}, we can state the following 
second order conditions.

\begin{proposition}[A Sufficient Condition via Slack Variables]\label{prop:sosc_sym}
  Let $(x,\mult,\lambda) \in \Re^n \times \Re^m \times \jAlg$ be a KKT triple of problem~\eqref{eq:sym}. Suppose that
  \begin{equation}\label{eq:sosc_sym}
  \inner{\grad _x ^2 L(x,\mult,\lambda)v}{v} + 2\inner{\jProd{w}{w}}{\lambda} > 0, \tag{SOSC-NSCP}
  \end{equation}
  for every nonzero $(v,w) \in \Re^n \times \jAlg$ such that $\jac{g}(x)v - 2\jProd{\sqrt{g(x)}}{w} = 0$ and 
  $\jac{h}(x)v = 0$. Then, $x$~is a local minimum for \eqref{eq:sym}, $\lambda \in \stdCone$, and strict complementarity is satisfied.
\end{proposition}
\begin{proof}
If $(x,\mult,\lambda)$ is a KKT triple for \eqref{eq:sym}, then 	$(x,\sqrt{g(x)},\mult,\lambda)$ is a 
KKT quadruple for \eqref{eq:slack}. Then, from Proposition~\ref{prop:sosc_slack}, we conclude 
that $x$ must be a local minimum. Taking $v = 0$ in \eqref{eq:sosc_sym}, we see that 
$$
\inner{\jProd{w}{w}}{\lambda} > 0,
$$
for all $w$ such that $\jProd{\sqrt{g(x)}}{w} = 0$. Due to Theorem~\ref{theo:strict}, we have $\lambda \in \stdCone$ and 
$\matRank \sqrt{g(x)} + \matRank \lambda = r$. As $\matRank\sqrt{g(x)} = \matRank {g(x)}$, we conclude 
that strict complementarity is satisfied.  
\end{proof}
Interestingly, the condition in Proposition~\ref{prop:sosc_sym} is strong enough to ensure strict complementarity. 
And, in fact, when strict complementarity holds and $\stdCone$ is either the cone of positive semidefinite matrices or a product of Lorentz cones, 
the condition in Proposition~\ref{prop:sosc_sym} is equivalent to the second order sufficient conditions described 
by Shapiro \cite{shapiro97} and Bonnans and Ram\'{i}rez \cite{BR05}. See \cite{FF16}, \cite{FF17} and \cite{LFF16} for more details.
We also have the following necessary condition.

\begin{proposition}[A Necessary Condition via Slack Variables]\label{prop:sonc_sym}
	Let $x \in \Re^n$ be a local minimum of \eqref{eq:sym}.
	Assume that $(x,\mult,\lambda)  \in \Re^n \times \Re^m \times \jAlg$ is a KKT triple for \eqref{eq:sym} satisfying 
	nondegeneracy. Then the following condition holds:
	\begin{equation}
	\inner{\grad _x ^2 L(x,\mu,\lambda)v}{v} + 2\inner{\jProd{w}{w}}{\lambda} \geq 0, \tag{SONC-NSCP}
	\end{equation}
	for every $(v,w) \in \Re^n \times \jAlg$ such that $\jac{g}(x)v - 2\jProd{\sqrt{g(x)}}{w} = 0$ and $\jac{h}(x)v = 0$.
\end{proposition} 
\begin{proof}
If $(x,\mult,\lambda)$ is a KKT triple for \eqref{eq:sym}, then 	$(x,\sqrt{g(x)},\mult,\lambda)$ is a 
KKT quadruple for \eqref{eq:slack}. Moreover, if $x$ is a local minimum for \eqref{eq:sym}, then 
$(x,y)$ is a local minimum for \eqref{eq:slack}. As $x$ satisfies nondegeneracy, LICQ is 
satisfied at  $(x,y)$, so that we are under the hypothesis of Proposition~\ref{prop:sonc_slack}.
\end{proof}

\section{A simple augmented Lagrangian method and its convergence}\label{sec:alg}

In \cite{N07}, Noll warns against naively extending algorithms for NLPs
to  nonlinear conic programming. One of the reasons is that those 
extensions often use unrealistic second order conditions which ignore the extra 
terms that appear in no-gap SOSCs for nonlinear cones. He then argues that such
conditions are unlikely to hold in practice. He goes 
 on to prove convergence results for an augmented Lagrangian method for  NSDPs based 
on the no-gap optimality conditions obtained by Shapiro \cite{shapiro97}.

We have already shown in \cite{LFF16} that if $\stdCone = \PSDcone{n}$, then Shapiro's SOSC for
\eqref{eq:sym} and the classical SOSC for  \eqref{eq:slack} are equivalent, under strict 
complementarity, see Propositions 10, 11, 13 and 14 therein. 
This suggests 
that it is viable to design appropriate algorithms for \eqref{eq:sym} by studying 
the NLP formulation \eqref{eq:slack} and  importing convergence results from nonlinear programming 
theory while avoiding the issues described by Noll. Furthermore, in some cases, we can remove the slack variable $y$ altogether from the final algorithm.
We will illustrate this approach by describing an augmented Lagrangian method for 
\eqref{eq:sym}.

Bertsekas also suggested a similar approach in \cite{Ber82a}, where he analyzed augmented Lagrangian methods
for \emph{inequality} constrained NLPs by first reformulating them as \emph{equality} constrained 
NLPs with the aid of squared slack variables. Kleinmichel and Sch\"onefeld described a method for NLPs in \cite{KS88} where squared slack variables were used not only to deal with the inequality constraints but in place of the Lagrange multipliers, as a way to force them to be nonnegative.  More recently, Sun, Sun and Zhang \cite{SSZ08} showed how to obtain a 
convergence rate result for an augmented Lagrangian method for NSDPs using slack variables, under the hypothesis of 
strict complementarity\footnote{We remark that their main contribution was to show a convergence rate result using the strong second order sufficient condition, nondegeneracy but without strict complementarity.}, 
see Theorem 3 therein. Here we take a closer look at this topic and 
extend their Theorem 3.

\subsection{Augmented Lagrangian method for \eqref{eq:sym}}
\label{sec:methods}

Let $\varphi: \Re^n \to \Re$, $\psi: \Re^n \to \Re^m$ be twice differentiable 
functions and consider the following NLP:
\begin{equation}
  \label{eq:nlp}
  \begin{array}{ll}
    \underset{x}{\mbox{minimize}} & \varphi(x) \\ 
    \mbox{subject to} & \psi(x) = 0.
  \end{array}
\end{equation}
Following Section 17.3 in \cite{NW99}, given a multiplier $\lambda \in \Re^m$ and a penalty parameter $\augP \in \Re$, define the augmented Lagrangian $\aug{\augP}:\Re^n\times \Re^m \to \Re$ for \eqref{eq:nlp} by 
$$
\aug{\augP} (x,\lambda) = \varphi(x )- \inner{\psi(x)}{\lambda} + \frac{\augP}{2}\norm{\psi(x)}^2.
$$
For problem \eqref{eq:slack}, the augmented Lagrangian is given by 
\begin{equation}
\augSlack (x,y, \mult,\lambda) = f(x) - \inner{h(x)}{\mult} + \frac{\augP}{2}\norm{h(x)}^2 - \inner{g(x) - \jProd{y}{y}}{\lambda} +    \frac{\augP}{2}\norm{g(x) - \jProd{y}{y}}^2. \label{eq:aug}
\end{equation}
We  have the following basic augmented Lagrangian method.

\begin{algorithm}[H]
	\caption{Augmented Lagrangian Method for \eqref{eq:slack} }\label{alg:aug_slack}
	
	Choose  initial points $x_1, y_1$, initial multipliers $\mult _1, \lambda _1$ and an initial penalty $\augP_1$. \\
	$k \leftarrow 1$. \\
	Let $(x_{k+1},y_{k+1})$ be a minimizer of $\augSlackP{\augP_k}(\cdot,\cdot, \mult _k, \lambda _k)$.  \label{alg:aug_slack:it} \\
	$\mult _{k+1} \leftarrow \mult _k - \augP_k h(x_{k+1})$.\\
	 $\lambda _{k+1} \leftarrow \lambda _k -  \augP_k (g(x_{k+1}) - \jProd{y_{k+1}}{y_{k+1}} )$. \label{alg:aug_slack:mult}\\		
	Choose a new penalty $\augP_{k+1} $ with $\augP_{k+1} \geq \augP_{k}$.\\
	Let $k \leftarrow k + 1$ and return to Step \ref{alg:aug_slack:it}.
\end{algorithm}

We will show  how to remove the slack variable from the augmented Lagrangian.

\begin{proposition}\label{prop:alg}
The following equation holds:
\begin{equation}
\min _{y} \augSlack (x,y,\mu, \lambda) = f(x)  - \inner{h(x)}{\mu} + \frac{\augP}{2}\norm{h(x)}^2  +\frac{1}{2\augP}\left(- \norm{\lambda}^2+    
\norm{\proj{ \lambda -\augP g(x)}}^2 \right).\label{eq:aug_sym}
\end{equation}
Moreover if $(x^*,y^*)$ is a minimum of $\augSlack (\cdot,\cdot,\mu, \lambda)$, then 
$
\jProd{y^*}{y^*} =  \proj{g(x^*) - \frac{\lambda}{\augP}}.
$
\end{proposition}

\begin{proof}
In the partial minimization $\min _{y} \augSlack (x,y, \mu, \lambda)$, we look at the terms that depend on $y$.
\begin{align*}
\min _{y}\left(-\inner{g(x) - \jProd{y}{y}}{\lambda} + \frac{\augP}{2}\norm{g(x) - \jProd{y}{y}}^2\right)  
\end{align*}
 \vspace{-1.5\baselineskip}
\begin{flalign*}
 & &&=   \min _{y}\left(-\inner{g(x) - \jProd{y}{y}}{\lambda} + \frac{\augP}{2}\norm{g(x)  - \frac{\lambda}{\augP}- \jProd{y}{y}+ \frac{\lambda}{\augP}}^2\right)\\
&&&=  \min _{y}\left(  \frac{\augP}{2}
\norm{g(x)  - \frac{\lambda}{\augP}- \jProd{y}{y}}^2 - \frac{\norm{\lambda}^2}{2\augP} \right) \\
 & &&= - \frac{\norm{\lambda}^2}{2\augP}+  \frac{\augP}{2}\min _{y}
\norm{g(x)  - \frac{\lambda}{\augP} - \jProd{y}{y}}^2.   
\end{flalign*}
Note that 
\begin{equation}
\min _{y}
\norm{g(x)  - \frac{\lambda}{\augP} - \jProd{y}{y}}^2  = \min _{z \in \stdCone}
\norm{g(x)  - \frac{\lambda}{\augP} - z}^2. \label{eq:aug_aux}
\end{equation}
Then, \eqref{eq:aug_aux} together with Lemma~\ref{lemma:moreau}  implies that
\begin{flalign*}
 - \frac{\norm{\lambda}^2}{2\augP}+  \frac{\augP}{2}\min _{y}
 \norm{g(x)  - \frac{\lambda}{\augP} - \jProd{y}{y}}^2  &= - \frac{\norm{\lambda}^2}{2\augP}+    \frac{\augP}{2}
\norm{P_{\stdCone}\left( \frac{\lambda}{\augP}-g(x)  \right)}^2   \\
  &= - \frac{\norm{\lambda}^2}{2\augP}+    \frac{1}{2\augP}
\norm{\proj{ \lambda -\augP g(x)}}^2.
\end{flalign*}
It follows that
$$
\min _{y}  \augSlackP{\augP} (x,y, \mu, \lambda) = f(x)  - \inner{h(x)}{\mu} + \frac{\augP}{2}\norm{h(x)}^2  +\frac{1}{2\augP}\left(- \norm{\lambda}^2+    
\norm{\proj{ \lambda -\augP g(x)}}^2 \right)
$$
and hence \eqref{eq:aug_sym} holds.

Finally, note that \eqref{eq:aug_aux} implies that if $(x^*,y^*)$ 
is a  minimum of $\augSlack (\cdot,\cdot,\mu, \lambda)$, 
then $\jProd{y^*}{y^*} = \proj{  g(x^*) - \frac{\lambda}{\augP} }$. 
\end{proof}

Proposition~\ref{prop:alg} suggests  the following augmented Lagrangian for \eqref{eq:sym}:
\begin{equation}
\augCone(x,\mu,\lambda) = f(x)  - \inner{h(x)}{\mu} + \frac{\augP}{2}\norm{h(x)}^2  +\frac{1}{2\augP}\left(- \norm{\lambda}^2+\norm{\proj{ \lambda -\augP g(x)}}^2 \right). \label{eq:aug_cone}
\end{equation}

Moreover, due to Lemma~\ref{lemma:moreau} and Proposition~\ref{prop:alg}, we can write the 
multiplier update in Step \ref{alg:aug_slack:mult} of Algorithm \ref{alg:aug_slack} as 
$$
\lambda _{k+1} \leftarrow \proj{ \lambda _k -  \augP_k g(x_{k+1})}.
$$

This gives rise to the following augmented Lagrangian method for \eqref{eq:sym}. 
Note that the squared slack variable $y$ is absent.

\begin{algorithm}[H]
	\caption{Augmented Lagrangian Method for \eqref{eq:sym} }\label{alg:aug_sym}
	
	Choose  an initial point $x_1$, initial multipliers $\mu _1, \lambda _1$ and an initial penalty $\augP_1$. \\
	$k \leftarrow 1$. \\
	Let $x_{k+1}$ be a minimizer of $\augConeP{\augP_k}(\cdot, \mu _k, \lambda _k)$.  \label{alg:aug_sym:it} \\
	$\mu _{k+1} \leftarrow \mu _k - \augP_k h(x_{k+1})$.\\
	$\lambda _{k+1} \leftarrow \proj{ \lambda _k -  \augP_k g(x_{k+1})}$. \\		
	Choose a new penalty  $\augP_{k+1} $ with $\augP_{k+1} \geq \augP_{k}$.\\
	Let $k \leftarrow k + 1$ and return to Step \ref{alg:aug_sym:it}.
\end{algorithm}
Note that Algorithms \ref{alg:aug_slack} and \ref{alg:aug_sym} are equivalent in the sense that any sequence of iterates 
$(x_k,y_k,\mu_k,\lambda _k)$ for Algorithm \ref{alg:aug_slack} is such that $(x_k,\mu_k,\lambda _k) $ is a valid sequence of 
iterates for Algorithm  \ref{alg:aug_sym}. Conversely, given a sequence $(x_k,\mu_k,\lambda _k) $ for Algorithm \ref{alg:aug_sym}, the sequence  
$(x_k, \sqrt{\proj{g(x_k) - \lambda/\augP_k}},\mu_k,\lambda _k)$ is valid for Algorithm \ref{alg:aug_slack}.
However, there could be computational differences between both algorithms. 
On the one hand, the subproblem in Algorithm \ref{alg:aug_sym} has less variables than the subproblem in Algorithm \ref{alg:aug_slack}. 
On the other hand, if $f,h,g$ are twice differentiable, the same is true for $\augSlack(\cdot,\cdot,\mu,\rho) $, while  $\augCone(\cdot,\mu,\rho)$ will not necessarily be twice differentiable in general. 
In this sense, the subproblem in Algorithm \ref{alg:aug_slack} is smoother than the one in Algorithm~\ref{alg:aug_sym}. 
In the appendix, we describe some preliminary numerical experiments aimed at exploring the difference between both approaches.

Note also that when $\stdCone$ is the cone of positive semidefinite matrices or a product of second order cones, Algorithm \ref{alg:aug_sym} gives 
 exactly the same augmented Lagrangian method with quadratic penalty discussed extensively in the literature. 
This is because, due to Proposition~\ref{prop:proj}, the projection
$\proj{ \lambda _k -  \augP_k g(x_{k+1})}$ is just the result of zeroing the negative eigenvalues of $ \lambda _k -  \augP_k g(x_{k+1})$.

\subsection{Convergence results}
Here, we will reinterpret a result of \cite{Ber82a}. We will then use it to prove an analogous theorem for \eqref{eq:sym}. 
This extends Theorem~3 in \cite{SSZ08} for all nonlinear symmetric cone programs. 

\begin{proposition}[Proposition 2.4 in \cite{Ber82a}]\label{prop:aug_conv_slack}
Suppose that $(x^*,y^*,\mult^*,\lambda^*) \in  \Re^n \times \jAlg \times \Re^m \times \jAlg$ is a KKT quadruple for \eqref{eq:slack} such that 
\begin{itemize}
	\item \eqref{eq:sosc_nlp.1} is satisfied,
	\item LICQ is satisfied.
\end{itemize}
Moreover, let $\hat{\augP}$ be such that $\grad ^2 \augSlackP{\hat{\augP}}$ is positive definite\footnote{Such a $\hat{\augP}$ always exists, see the remarks before Proposition~2.4 in 
\cite{Ber82a}.}. Then there 
are positive scalars $\hat \delta, \hat \epsilon, \hat M$ such that
\begin{enumerate}
\item for all $(\mult, \lambda, \augP)$ in the set $\hat D \coloneqq \{ (\mult, \lambda, \rho) \mid \abs{\mult - \mult^*} +   \abs{\lambda - \lambda^*} < \hat \delta \augP, \hat{\augP} \leq \augP  \}$, the following problem 
has a unique solution:
\begin{equation}
\label{}
\begin{array}{ll}
\underset{x,y}{\mbox{minimize}} & \augSlackP{{\augP}}(x,y,\mult,\lambda) \label{eq:local_al_slack} \\ 
\mbox{subject to} & (x,y) \in \ball{x^*}{\hat \epsilon}\times \ball{y^*}{\hat \epsilon},
\end{array}
\end{equation}
where  $\ball{x^*}{\hat \epsilon} \subseteq \Re^n$ and $\ball{y^*}{\hat \epsilon} \subseteq \jAlg$ are the spheres with radius $\hat \epsilon$ centered at $x^*$ and 
$y^*$, respectively.
Denote such a solution by $(x(\mult,\lambda,\augP),y(\mult,\lambda,\augP))$. Then, $(x(\cdot,\cdot,\cdot), y(\cdot,\cdot,\cdot))$ is continuously differentiable in 
the interior of $\hat D$ and satisfies
\begin{equation}
\abs{(x(\mult,\lambda,\augP),y(\mult,\lambda,\augP))-(x^*,y^*)} \leq \frac{\hat M}{\augP}\abs{(\mult, \lambda) - (\mult^*, \lambda^*)}, \label{prop:aug_conv_slack:1}
\end{equation}
 for all $(\mult, \lambda, \augP) \in \hat D$.
 
\item For all $(\mult, \lambda, \augP) \in \hat D$, we have
\begin{equation}
\abs{(\hat \mult(\mult,\lambda,\augP),\hat \lambda(\mult,\lambda,\augP))-(\mult^*, \lambda^*)} \leq \frac{\hat M}{\augP}\abs{(\mult, \lambda) - (\mult^*, \lambda^*)}, \label{prop:aug_conv_slack:2}
\end{equation}
where
\begin{align*}
\hat \mult(\mult,\lambda,\augP) & =  \mult - \augP h(x(\mult,\lambda,\augP)), \\
\hat \lambda(\mult,\lambda,\augP) & = \lambda -\augP(g(x(\mult,\lambda,\augP)) - y(\mult,\lambda,\augP)^2 ).
\end{align*}

\end{enumerate}
\end{proposition}

Our goal is to prove the following result.
\begin{proposition}\label{prop:aug_conv_sym}
	Suppose that $(x^*,\mult^*,\lambda^*) \in  \Re^n \times \Re^m \times \jAlg$ is a KKT triple for \eqref{eq:sym} such that 
	\begin{itemize}
		\item \eqref{eq:sosc_sym} is satisfied,
		\item Nondegeneracy (see Definition \ref{def:nondeg}) is satisfied.
	\end{itemize}
Then there 
	are positive scalars $\delta, \epsilon, M, \overline \augP$ such that
	\begin{enumerate}
		\item For all $(\mult, \lambda, \augP)$ in the set $D \coloneqq \{ (\mult, \lambda, \rho) \mid \abs{\mult - \mult^*} +   \abs{\lambda - \lambda^*} < \delta \augP, \overline{\augP} \leq \augP  \}$, the following problem 
		has a unique solution:
		\begin{align}
		\underset{x}{\text{minimize}} & \quad \augConeP{\augP}(x,\mult,\lambda) \label{eq:theo_min2}\\ 
		\text{subject to} &\quad  x \in \ball{x^*}{\epsilon}. \nonumber
		\end{align}

		Denote such a solution by $x(\mult,\lambda,\augP)$. Then, $x(\cdot,\cdot,\cdot)$ is continuously differentiable in 
		the interior of $D$ and satisfies
		\begin{equation}
		\abs{x(\mult,\lambda,\augP)-x^*} \leq \frac{M}{\augP}\abs{(\mult, \lambda) - (\mult^*, \lambda^*)} \label{prop:aug_conv_sym:1}
		\end{equation}
		for all $(\mult, \lambda, \augP) \in D$.
		
		\item For all $(\mult, \lambda, \augP) \in D$, we have
		\begin{equation}
		\abs{(\overline \mult(\mult,\lambda,\augP),\overline  \lambda(\mult,\lambda,\augP))-(\mult^*, \lambda^*)} \leq \frac{M}{\augP}\abs{(\mult, \lambda) - (\mult^*, \lambda^*)}, \label{prop:aug_conv_sym:2}
		\end{equation}
		where
		\begin{align*}
		\overline  \mult(\mult,\lambda,\augP) & =  \mult - \augP h(x(\mult,\lambda,\augP)), \\
		\overline  \lambda(\mult,\lambda,\augP) & =  \proj{ \lambda -  \augP g(x(\mult,\lambda,\augP))}.
		\end{align*}
		
	\end{enumerate}
\end{proposition}

Note that the argument in Proposition~\ref{prop:alg} shows that \eqref{eq:theo_min2} is equivalent to
\begin{align}
\underset{x,y}{\text{minimize}} & \quad\augSlackP{{\augP}}(x,y,\mult,\lambda) \label{eq:theo_min1_aux}\\ 
\text{subject to} &\quad  (x,y) \in \ball{x^*}{\epsilon}\times  \jAlg. \nonumber
\end{align}
Moreover, if $(\hat x, \hat y)$ is an optimal solution for  problem \eqref{eq:theo_min1_aux},  we have $\hat y^2 = \proj{g(\hat x) - {\lambda}/{\augP}}$. 
Therefore, if $\hat \delta$ and $\hat \epsilon$ provided by Proposition~\ref{prop:aug_conv_slack} were such that  
$ \sqrt{\proj{g(x) - {\lambda}/{\augP}}}$ stays in the ball $\ball{\sqrt{g(x^*)}}{\hat \epsilon}$ for all $x \in \ball{x^*}{\hat \epsilon}$, then it would
be very 
straightforward to prove Proposition~\ref{prop:aug_conv_sym}. As this  is not generally the case, in the  
proof we have to argue that we can adjust $\hat \delta,\hat \epsilon$ appropriately.  
The proof  is similar to \cite{SSZ08}, but we need to adjust the proof to make use of the Euclidean Jordan algebra machinery.
Before we proceed, we need a few auxiliary lemmas.
\begin{lemma}\label{lemma:root}
Let $\jAlg$ be an Euclidean Jordan algebra with rank $r$ and let $\psi _ \jAlg$ denote 
the function that maps $y$ to $\sqrt{\abs{y}}$, where $\abs{y} = \sqrt{y^2}$. Let $\jAlg^* = \{y \in \jAlg \mid \matRank y = r\}$.
Then, $\psi _ \jAlg$ is continuously differentiable over $\jAlg^*$.
\end{lemma}
\begin{proof}
Let $Q = \{x \in \Re^r \mid x_i \neq 0, \forall i\}$ and $\varphi:\Re^r \to \Re^r$ be the function 
that maps $x\in \Re^r$ to $(\sqrt{\abs{x_1}},\ldots,\sqrt{\abs{x_r}})$. Then, following the discussion 
in Section 6 of \cite{Baes04}, $\psi 	_ \jAlg$ is the spectral map generated 
by $\varphi$. That is, if $y \in \jAlg$ and its spectral decomposition is 
given by $y = \sum _{i=1}^r \sigma _i c_i$, then 
$$
\psi 	_ \jAlg(y) = \sum _{i=1}^r \varphi_i(\sigma _i) c_i,
$$
where  $\varphi _i: \Re \to \Re$ are the component functions of $\varphi$. Then, Theorem~53 in \cite{Baes04} shows 
that $\psi _ \jAlg$ is continuously differentiable over $\jAlg^*$, because 
$\varphi $ is continuously differentiable over $Q$. A similar conclusion also follows from 
Theorem~3.2 in \cite{SS08}.
\end{proof}

\begin{lemma}\label{lemma:half}
	Let $\jAlg$ be an Euclidean Jordan algebra with rank $r$, let $c$ be 
	an idempotent and $w \in V(c,1/2)$. Then there are $w_{0} \in V(c,0)$, 
	$w_{1} \in V(c,1)$ such that 
	\begin{align*}
	w^2 & = w_{0}^2 + w_{1}^2, \\
	\norm{w_{0}^2} & = \norm{w_{1}^2}.
	\end{align*}
\end{lemma}
\begin{proof}
According to Theorem \ref{theo:peirce1}, $w^2 \in V(c,0) + V(c,1)$. 
As $w^2 \in \stdCone$, this implies the existence of $w_0 \in V(c,0)$ 
and $w_1 \in V(c,1)$ such that  $w ^2 = w_{0}^2 + w_{1}^2$. 
From the proof of Proposition~IV.1.1 in~\cite{FK94}, we 
see that, in fact,
$$
w_{0}^2 = w^2 - \jProd{c}{w^2}, \qquad w_{1}^2 = \jProd{c}{w^2}.
$$
Then, we have
\begin{align*}
\norm{w_{0}^2}^2 & = \norm{w^2}^2 - 2\inner{w^2 }{\jProd{c}{w^2}} + \norm{w_1^2}^2\\
& = \norm{w^2}^2 - 2\inner{w^3 }{\jProd{c}{w}} + \norm{w_1^2}^2\\
& = \norm{w^2}^2 - \inner{w^3 }{w} + \norm{w_1^2}^2 \\
& = \norm{w_1^2}^2,
\end{align*}
where the second equality follows from the power associativity of the Jordan 
product and the fact that the algebra is Euclidean, which implies that 
$\inner{w^2 }{\jProd{c}{w^2}} =  \inner{\jProd{w^3}{w}}{c} = \inner{w^3}{\jProd{c}{w}} $. 
Then, the third equality follows from $w \in V(c,1/2)$.
\end{proof}

\begin{proof}[Proof of Proposition \ref{prop:aug_conv_sym}]
Note that if \eqref{eq:sosc_sym} and nondegeneracy are satisfied at $(x^*,\mult^*,\lambda^*)$, then 
$ (x^*,\sqrt{g(x^*)},\mult^*,\lambda^*)$ satisfies \eqref{eq:sosc_nlp.1} and LICQ, due 
to Corollary \ref{col:licq_nondeg}. 
So let $\hat{\augP}, \hat \delta, \hat \epsilon, \hat M$ and $\hat D$ be as 
in Proposition~\ref{prop:aug_conv_slack}.
	
First, we consider the spectral decomposition of $g(x^*)$. Without loss of 
generality, we may assume~that
$$
g(x^*) = \sum _{i = 1}^{\matRank g(x^*)} \eig _i c_i,
$$
where the first $\matRank g(x^*)$ eigenvalues are positive and the remaining are zero. 
Then, we let $c = c_1 + \cdots + c_{\matRank g(x^*)}$ and consider the Euclidean 
Jordan algebra $V(c,1)$ together with its associated symmetric cone 
$\stdFace = \{\jProd{w}{w} \mid w \in V(c,1) \}$. We know that $V(c,1)$ has rank equal to $\matRank g(x^*)$. Moreover, by Proposition~\ref{prop:aux}, $g(x^*)$ belongs to the relative interior of $\stdFace$. In particular, we may select $\hat \epsilon _1 \in (0,\hat \epsilon)$ such 
that $\norm{z-g(x^*)} \leq \hat \epsilon _1$ and $z \in V(c,1)$ implies that 
$z$ lies in the relative interior of $\stdFace$ as well.

Now, we take the function $\psi _ {V(c,1)}$ from Lemma~\ref{lemma:root}. Note that 
if $v \in \stdCone$ , then $\psi _ {V(c,1)}(v) = \sqrt{v}$.
Since $\psi _ {V(c,1)}$ is continuously 
differentiable in $V(c,1)^*$, the mean value inequality tells us that for $\norm{v-g(x^*)} \leq \hat \epsilon _1$
and $v \in V(c,1)$ we have
\begin{align}
\norm{\sqrt{v} - \sqrt{g(x^*)}} \leq R\norm{v - g(x^*)},\label{eq:al_root}
\end{align} 
where $R$ is the supremum of $\norm{\jac{ \psi _{V(c,1)}}}$ over 
the set $V(c,1) \cap \ball{g(x^*)}{\hat \epsilon _1}$. We then let $\hat \epsilon _2 \in (0,\epsilon _1]$ be such that 
\begin{equation}
4 R^2\hat \epsilon_2^2 + \hat \epsilon_2 \sqrt{2r} + \hat \epsilon_2 \sqrt{r} \leq \hat \epsilon _1^2.\label{eq:al_e1_e2}
\end{equation}

Since $g$ is continuously differentiable, again by the mean value inequality,
 there is $l_g$ such that for 
every $x \in \ball{x^*}{\hat{\epsilon}}$ we have
\begin{equation}\label{eq:uniform_cont}
\norm{g(x) - g(x^*)} \leq l_g \norm{x- x^*}.
\end{equation}

We are now ready to construct the neighborhood $D$. We select $\epsilon, \delta, \overline \rho, M$ such that the 
following conditions are satisfied:
\begin{enumerate}
	\item $\epsilon \in (0,\hat \epsilon ], \delta \in (0, \hat \delta], \overline \rho \geq \hat \rho, M \geq \hat M$. \label{cond:1}
	\item $l_g \epsilon + \delta +  \norm{\frac{\lambda^*}{\overline \augP}} \leq \hat \epsilon _2$. \label{cond:2}
	\item $\epsilon, \delta$ are small enough and $M, \overline \augP$ are large enough such that the conclusions of Proposition~\ref{prop:aug_conv_slack} hold for those $\epsilon, \delta, M, \overline \augP$ and 
	such that the neighborhood $\ball{x^*}{\hat \epsilon}\times \ball{y^*}{\hat \epsilon}$ in \eqref{eq:local_al_slack} can 
	be replaced by $\ball{x^*}{\epsilon}\times \ball{y^*}{\hat \epsilon_1}$ without affecting 
	the conclusion of the theorem. \label{cond:3}
\end{enumerate}
We then have 
$$
D = \{ (\mult, \lambda, \rho) \mid \abs{\mult - \mult^*} +   \abs{\lambda - \lambda^*} < \delta \augP, \overline{\augP} \leq \augP  \}.
$$
For all $(\mult, \lambda, \augP) \in D$ and $x$ such that 
$\norm{x-x^*}\leq \epsilon$, we have
\begin{align*}
\norm {\proj{g(x) - \frac{\lambda}{\augP}}-g(x^*)}  & = \norm {\proj{g(x) - \frac{\lambda}{\augP}}-\proj{g(x^*)}} \\
& \leq \norm{g(x) - g(x^*) - \frac{\lambda - \lambda^*}{\augP} -\frac{\lambda^*}{\augP} } \\
& \leq \norm{g(x) - g(x^*)} + \norm{\frac{\lambda - \lambda^*}{\augP}} +\norm{\frac{\lambda^*}{\augP}  }\\
& \leq l_g \norm{x- x^*} + \delta +  \norm{\frac{\lambda^*}{\augP}} \\
& \leq \hat \epsilon_2,
\end{align*}
where the first inequality follows from the fact that the projections are nonexpansive maps and the fourth inequality 
follows by the definition of $D$, \eqref{eq:uniform_cont} and the fact that $\overline{\augP} \leq \augP$.

Now we will show that $\norm {\sqrt{\proj{g(x) - {\lambda}/{\augP}}}-\sqrt{g(x^*)}} \leq \hat \epsilon_1$ holds as well, which is our primary goal. More generally, we will show 
that if $v \in \stdCone \cap \ball{g(x^*)}{\hat \epsilon _2}$ then 
$\sqrt{v} \in \ball{\sqrt{g(x^*)}}{\hat \epsilon _1}$.

Thus suppose that $v \in \stdCone$ is such that $v \in \ball{g(x^*)}{\hat \epsilon _2}$. We consider 
the Peirce decomposition of $\sqrt{v}$ with respect the idempotent $c$, as in Theorem~\ref{theo:peirce1}. We write 
$\sqrt{v} = w_1 + w_2 + w_3$, where $w_1 \in V(c,1), w_2 \in V(c,1/2), w_3 \in V(c,0)$. We then have 
$v = w_1^2 + w_2^2 + w_3^2 +  2\jProd{w_2}{(w_1+w_3)} $. Also by Theorem~\ref{theo:peirce1}, $2\jProd{w_2}{(w_1+w_3)} \in V(c,1/2)$ and 
	$w_2^2 = w_{2,1}^2 + w_{2,0}^2$, for some $ w_{2,1} \in  V(c,1)$ and $w_{2,0} \in V(c,0)$. We group 
	the terms of $v-g(x^*)$ as follows
\begin{equation}
v-g(x^*) = ( w_1^2 + w_{2,1}^2 - g(x^*) ) + (w_3^2  + w_{2,0}^2) + (2\jProd{w_2}{(w_1+w_3)}),
\end{equation}
where the terms in parentheses belong, respectively, to the mutually orthogonal subspaces $V(c,1)$, $V(c,0)$ and $V(c,1/2)$.
Therefore, $\norm{v-g(x^*)}^2\leq \hat\epsilon _2^2$ implies that
\begin{align}
\norm{w_1^2 + w_{2,1}^2 - g(x^*)}^2 & \leq \hat \epsilon_2^2, \label{eq:al_1}\\ 
\norm{w_3^2}^2 & \leq \hat \epsilon_2^2, \label{eq:al_3}\\
\norm{w_{2,0}^2}^2 & \leq \hat \epsilon_2^2, \label{eq:al_4}
\end{align} 
where the last two inequalities follows from the fact that  $\norm{w_3^2  + w_{2,0}^2}^2 =  \norm{w_3^2}^2 + \inner{w_3^2}{w_{2,0}^2}+ \norm{w_{2,0}^2}^2  \leq \hat \epsilon_2^2$ and 
that $\inner{w_3^2}{w_{2,0}^2} \geq 0$, since $w_3^2,w_{2,0}^2 \in \stdCone$. From 
\eqref{eq:al_1}, \eqref{eq:al_4} and Lemma~\ref{lemma:half}, we obtain
$$
\norm{w_1^2 - g(x^*)}  \leq \norm{w_1^2 + w_{2,1}^2 - g(x^*)} +  \norm{w_{2,1}^2} \label{eq:al_5} \leq 2\hat \epsilon_2.\\ 
$$

We then use \eqref{eq:al_root} to conclude that 
\begin{equation}
\norm{w_1 - \sqrt{g(x^*)}} \leq 2 R\hat \epsilon_2 \label{eq:al_w1_g},
\end{equation}
since $\sqrt{w_1^2} = w_1$ holds\footnote{If we fix $v \in \stdCone$, there might be several elements $a$ satisfying $\jProd{a}{a} = v$, but only 
	one of those $a$ belongs to $\stdCone$. With our definition, $\sqrt{v}$ is precisely this $a$. As $\sqrt{v} \in \stdCone$, the element $w_1$ appearing 
	in its Pierce decomposition also belongs to $\stdCone$. To see that, note that if $w_1$ had a negative eigenvalue $\sigma$, we would have $\inner{\sqrt{v}}{\sigma d} = \inner{w_1}{\sigma d } < 0$, where $d$ is the idempotent associated to $\sigma$. This shows that  $\sqrt{w_1^2} = w_1$ indeed. }.

Recall that given $z \in \stdCone$, we have $\norm{z}^2 = \sigma _1^2 + \cdots + \sigma _r^2$, 
where $\sigma _i$ are the eigenvalues of $z$. Therefore, $\norm{z}^2$ is the $1$-norm of the vector $u = (\sigma _1^2, \ldots, \sigma _r^2)$, which is majorized 
by $\norm{u}_2 \sqrt{r} = \norm{z^2}\sqrt{r}$, i.e., $\norm{z}^2 \leq \norm{z^2} \sqrt{r} $. This, together with \eqref{eq:al_3}, \eqref{eq:al_4} and Lemma~\ref{lemma:half} imposes the following inequalities on 
the $w_i$:
\begin{align} 
\norm{w_3}^2 & \leq \hat \epsilon_2 \sqrt{r},  \label{eq:al_6} \\ 
\norm{w_2}^2 \leq  \norm{w_2^2} \sqrt{r} =    \norm{w_{2,0}^2} \sqrt{2r} & \leq \hat \epsilon_2 \sqrt{2r}. \label{eq:al_7}
\end{align} 
From $\sqrt{g(x^*)} \in V(c,1)$, \eqref{eq:al_w1_g}, \eqref{eq:al_6} and \eqref{eq:al_7}, we obtain
\begin{align*}
\norm{\sqrt{v} - \sqrt{g(x^*)}}^2 & = \norm{w_1-\sqrt{g(x^*)}}^2 +  \norm{w_2}^2 + \norm{w_3}^2 \\
& \leq 4 R^2\hat \epsilon_2^2 + \hat \epsilon_2 \sqrt{2r} + \hat \epsilon_2 \sqrt{r} \\
& \leq  \hat \epsilon_1^2, 
\end{align*}
where the last inequality follows from \eqref{eq:al_e1_e2}. 
To recap, 
we have shown that whenever $(\mult,\lambda,\augP) \in D$ and $x$ is such that $\norm{x - x^*} \leq \epsilon$, 
then $$
{\sqrt{\proj{g(x) - \frac{\lambda}{\augP}}}} \in \ball{\sqrt{g(x^*)}}{\hat \epsilon _1}. 
		$$
Thus, letting $x(\mult,\lambda,\augP)$ be a minimizer of \eqref{eq:theo_min2} and $\hat y = {\sqrt{\proj{g(x(\mult,\lambda,\augP)) - {\lambda}/{\augP}}}}$, we have $x(\mult,\lambda,\augP) \in \ball{x^*}{\epsilon} $, and $\hat y \in \ball{\sqrt{g(x^*)}}{\hat \epsilon _1}$. 
Now, we note that $(x(\mult,\lambda,\augP),\hat y)$ is a minimizer of \eqref{eq:local_al_slack} with 
$y^* = \sqrt{g(x^*)}$ and $\ball{x^*}{\epsilon}\times \ball{y^*}{\hat \epsilon_1}$ in place of $\ball{x^*}{\hat \epsilon}\times \ball{y^*}{\hat \epsilon}$. 
In fact, let $(\overline x, \overline y)$ be the minimizer of \eqref{eq:local_al_slack}. By Proposition 
\ref{prop:alg}, we have 
\begin{align*}
\augSlackP{\augP}(x(\mult,\lambda,\augP),\hat y, \mu,\lambda)& = \augConeP{\augP}(x(\mult,\lambda,\augP),\mu,\lambda) \\ 
& \leq \augConeP{\augP}(\overline x,\mu,\lambda) \\
& = \min _{y}  \augSlackP{\augP} (\overline x, y, \mu, \lambda)\\
& \leq \augSlackP{\augP} (\overline x, \overline y, \mu, \lambda)
\end{align*}
As $(x(\mult,\lambda,\augP),\hat y)$ is feasible to \eqref{eq:local_al_slack}, we conclude that $(x(\mult,\lambda,\augP),\hat y)$ is indeed a minimizer of \eqref{eq:local_al_slack} with 
$\ball{x^*}{\epsilon}\times \ball{y^*}{\hat \epsilon_1}$ in place of $\ball{x^*}{\hat \epsilon}\times \ball{y^*}{\hat \epsilon}$. Furthermore, since the minimizer of \eqref{eq:local_al_slack} is unique, it follows 
that  $x(\mult,\lambda,\augP)$ is the unique minimizer of 
 $\eqref{eq:theo_min2}$. Due to  Proposition~\ref{prop:aug_conv_slack} and the choice of $\epsilon, \delta, M, \overline \augP$ (see items \ref{cond:1} to \ref{cond:3} above),  $x(\cdot,\cdot,\cdot)$ must be differentiable in the interior of $D$.
 Since \eqref{prop:aug_conv_slack:1} holds, then \eqref{prop:aug_conv_sym:1} holds as well. This concludes 
 item 1.
 
 Item $2$ also follows from the fact that 
 \begin{align*}
 \hat \lambda(\mult,\lambda,\augP) & = \lambda - \augP(g(x(\mult,\lambda,\augP)) - \hat y^2)  \\ 
 & =  \lambda - \augP\left( g(x(\mult,\lambda,\augP)) - \proj{g(x(\mult,\lambda,\augP)) - \frac{\lambda}{\augP}}\right) \\
 & = \proj{\augP g(x(\mult,\lambda,\augP)) - \lambda} - (\augP g(x(\mult,\lambda,\augP)) - \lambda ) \\
 & = \proj{\lambda - \augP g(x(\mult,\lambda,\augP)) }\\
 & = \overline \lambda(\mult,\lambda,\augP),
 \end{align*}
 where the  second to last equality follows from  Lemma~\ref{lemma:moreau}. Hence the 
 estimate in \eqref{prop:aug_conv_slack:2} also implies the estimate in \eqref{prop:aug_conv_sym:2}.
\end{proof}

\section{Concluding remarks}\label{sec:conc}
In this paper we presented a discussion on optimality conditions for nonlinear symmetric cone 
programs through slack variables. By doing so, we obtain an ordinary nonlinear programming problem, 
which is more straightforward to analyze. This connection gives some interesting insights, such 
as Theorem~\ref{theo:strict}, and makes it possible to analyze algorithms for \eqref{eq:sym} as we 
did in Section \ref{sec:alg}.

However, one slightly upsetting part of the usage of slack variables is the fact that when the 
second order sufficient conditions (SOSCs) are written down, we get that strict complementarity is 
automatically satisfied. In particular, if we have a KKT tuple that \emph{does not} 
satisfy strict complementarity and we wish to check whether it is a local minimum, there is no way to apply the theory described in this paper.
Of course, it is indeed possible to derive  SOSCs without assuming 
strict complementarity, as it was done in \cite{BR05, shapiro97}. However, at this point we do not know 
how to explain why this difference arises. An interesting research topic would be to find out whether \eqref{eq:sym} admits another reformulation as a 
nonlinear programming problem without this deficiency.

\section*{Appendix }

As we discussed in Section~\ref{sec:methods},
Algorithms~\ref{alg:aug_slack} and~\ref{alg:aug_sym} are equivalent in
terms of generated sequences, but there should be some differences
from the computational point of view. To illustrate this, we
implemented both algorithms in MATLAB R2016b, under Ubuntu 16.04, with
core i7 3.9GHz and 16GB of RAM. The numerical experiments here are
simple, so the subproblems were solved simply with the default
\texttt{fminunc} function, that uses a quasi-Newton method. Noting
that the sequences generated by both algorithms are the same only if
the subproblems are solved exactly, we set up the optimization
tolerance for the subproblems as $10^{-8}$. We also considered some
ideas from the augmented Lagrangian solver ALGENCAN (see Chapter~12 in~\cite{BM14}), so
the initial penalty parameter is dependent of the problem, and its
update is done by using an analysis of the infeasibility
measure. Moreover, the termination criteria consist in satisfying KKT
conditions approximately, with precision~$10^{-4}$.

We compare the algorithms with three problems, which are special cases
of NSCP: nonlinear programming problem (NLP) number 1 from Hock and
Schittkowski's list~\cite{HS80} (and originally by
Betts~\cite{Bet77}), nonlinear second-order cone programming problem (NSOCP)
by Kanzow, Ferenczi and Fukushima~\cite{KFF09}, and nonlinear semidefinite programming problem
(NSDP) by Noll~\cite{N07} (see Section~12 therein). The results are described in
Table~\ref{tab:experiments}, with ``inner iterations'' as the total
number of iterations for solving the subproblems. All problems were
successfully solved, using both algorithms. Also, for each problem,
the sequences of iterates $\{ x_k \}$ generated by the algorithms were
the same. However, the number of inner iterations and the total
times were clearly different. As expected, the slack variables approach
tends to require more time, because of the increase of the dimension
of the problem.

\begin{table}[!htb]
  \centering
  \caption{Comparison of Algorithms~\ref{alg:aug_slack} (with slack) 
    and~\ref{alg:aug_sym} (without slack).}
  \label{tab:experiments}
  \begin{tabular}{|c|c|c|c|c|}
     \hline
     \multirow{2}{*}{Problems} & \multirow{2}{*}{Methods} 
     & Outer & Inner & \multirow{2}{*}{Time (s)} \\
     & & iterations ($k$) & iterations & \\
    \hline
    \multirow{2}{*}{NLP}   & Algorithm~\ref{alg:aug_slack} & 1 & 50  & 0.017317 \\
                           & Algorithm~\ref{alg:aug_sym}   & 1 & 51  & 0.014929 \\
    \hline
    \multirow{2}{*}{NSOCP} & Algorithm~\ref{alg:aug_slack} & 5 & 76  & 0.069946 \\
                           & Algorithm~\ref{alg:aug_sym}   & 5 & 18  & 0.018908 \\ 
    \hline
    \multirow{2}{*}{NSDP}  & Algorithm~\ref{alg:aug_slack} & 5 & 230 & 0.210977 \\
                           & Algorithm~\ref{alg:aug_sym}   & 5 & 10  & 0.036631 \\
    \hline
  \end{tabular}
\end{table}

Whether there exist some cases where Algorithm 1 performs better is
still a matter of further investigation. In particular, the projection operator
used only in Algorithm 2 is computationally expensive for the NSDP case,
so depending on the optimization instance, the slack approach may be
advantageous. We note that Burer and Monteiro have proposed a very successful variant of 
Algorithm 1 for linear SDPs \cite{BM03,BM05}, where they exploit the existence of optimal solutions 
with low rank. In~\cite{LFF16}, for the NSDP case, 
we also presented a numerical comparison between an augmented Lagrangian method for 
the formulation \eqref{eq:sym} and for \eqref{eq:slack}. And, indeed, for some small 
problems, the formulation \eqref{eq:slack} had faster running times. However, 
we used the software PENLAB \cite{Penlab}, which uses a different kind of augmented Lagrangian
that is not covered by the discussion in Section \ref{sec:alg}.

\section*{Acknowledgments} 
We thank the referees for their helpful and insightful comments, which helped to improve the paper. 
In particular, the appendix  was included due to a suggestion by one of the referees.
We would also like to express our gratitude to Prof.~Andreas Fischer for bringing the reference \cite{KS88} to our attention.
This work was supported by Grant-in-Aid for Young Scientists (B) (26730012), for Scientific Research (C) (26330029) and for Scientific Research (B) (15H02968)  from Japan Society for the Promotion of Science.

\subsection*{Addendum }
\emph{This version corrects a mistake in the statement of nondegeneracy in Definition~2.8 (Definition~2.2 in the Mathematics of Operations Research journal version). Although the statement was wrong, the subsequent proofs and discussion that use nondegeneracy were done under the correct definition, so no other adjustments are necessary. We also adjusted the first sentence of Section~4 to replace ``local'' with ``global''.}

\bibliographystyle{abbrvurl}
\bibliography{journal-titles,references}

\begin{thebibliography}{10}

\bibitem{Baes04}
M.~Baes.
\newblock Convexity and differentiability properties of spectral functions and
  spectral mappings on {E}uclidean {J}ordan algebras.
\newblock {\em Linear Algebra and its Applications}, 422(2):664--700, 2007.

\bibitem{Ber82a}
D.~P. Bertsekas.
\newblock {\em Constrained Optimization and {L}agrange Multipliers Methods}.
\newblock Academic Press, New York, 1982.

\bibitem{Ber99}
D.~P. Bertsekas.
\newblock {\em Nonlinear Programming}.
\newblock Athena Scientific, 2nd edition, 1999.

\bibitem{Bet77}
J.~T. Betts.
\newblock An accelerated multiplier method for nonlinear programming.
\newblock {\em Journal of Optimization Theory and Applications},
  21(2):137--174, 1977.

\bibitem{BM14}
E.~G. Birgin and J.~M. Mart\'inez.
\newblock {\em Practical augmented {L}agrangian methods for constrained
  optimization}.
\newblock Society for Industrial and Applied Mathematics, Philadelphia, PA,
  2014.

\bibitem{BonnansCS}
J.~F. Bonnans, R.~Cominetti, and A.~Shapiro.
\newblock Second order optimality conditions based on parabolic second order
  tangent sets.
\newblock {\em SIAM Journal on Optimization}, 9(2):466--492, 1999.

\bibitem{BR05}
J.~F. Bonnans and H.~Ram\'irez~C.
\newblock Perturbation analysis of second-order cone programming problems.
\newblock {\em Mathematical Programming}, 104:205--227, 2005.

\bibitem{BS00}
J.~F. Bonnans and A.~Shapiro.
\newblock {\em Perturbation Analysis of Optimization Problems}.
\newblock Springer-Verlag, New York, 2000.

\bibitem{BM03}
S.~Burer and R.~D. Monteiro.
\newblock A nonlinear programming algorithm for solving semidefinite programs
  via low-rank factorization.
\newblock {\em Mathematical Programming}, 95(2):329--357, 2003.

\bibitem{BM05}
S.~Burer and R.~D. Monteiro.
\newblock Local minima and convergence in low-rank semidefinite programming.
\newblock {\em Mathematical Programming}, 103(3):427--444, 2005.

\bibitem{Cominetti1990}
R.~Cominetti.
\newblock Metric regularity, tangent sets, and second-order optimality
  conditions.
\newblock {\em Applied Mathematics and Optimization}, 21(1):265--287, 1990.

\bibitem{FK94}
J.~Faraut and A.~Korányi.
\newblock {\em Analysis on {S}ymmetric {C}ones}.
\newblock Oxford mathematical monographs. Clarendon Press, Oxford, 1994.

\bibitem{FB06}
L.~Faybusovich.
\newblock Jordan-algebraic approach to convexity theorems for quadratic
  mappings.
\newblock {\em SIAM Journal on Optimization}, 17(2):558--576, 2006.

\bibitem{FB08}
L.~Faybusovich.
\newblock Several {J}ordan-algebraic aspects of optimization.
\newblock {\em Optimization}, 57(3):379--393, 2008.

\bibitem{Penlab}
J.~Fiala, M.~Ko{\v c}vara, and M.~Stingl.
\newblock {PENLAB}: A matlab solver for nonlinear semidefinite optimization.
\newblock {\em ArXiv e-prints}, November 2013.
\newblock \href {http://arxiv.org/abs/1311.5240} {\path{arXiv:1311.5240}}.

\bibitem{FF16}
E.~H. Fukuda and M.~Fukushima.
\newblock The use of squared slack variables in nonlinear second-order cone
  programming.
\newblock {\em Journal of Optimization Theory and Applications},
  170(2):394--418, 2016.

\bibitem{FF17}
E.~H. Fukuda and M.~Fukushima.
\newblock A note on the squared slack variables technique for nonlinear
  optimization.
\newblock {\em Journal of the Operations Research Society of Japan},
  60(3):262--270, 2017.

\bibitem{HL93}
J.-B. Hiriart-Urruty and C.~Lemaréchal.
\newblock {\em Convex {A}nalysis and {M}inimization {A}lgorithms. {I}. ,
  Fundamentals}.
\newblock Grundlehren der mathematischen Wissenschaften. Springer, Berlin, New
  York, 1993.

\bibitem{HS80}
W.~Hock and K.~Schittkowski.
\newblock Test examples for nonlinear programming codes.
\newblock {\em Journal of Optimization Theory and Applications},
  30(1):127--129, 1980.

\bibitem{JS03}
H.~T. Jongen and O.~Stein.
\newblock On the complexity of equalizing inequalities.
\newblock {\em Journal of Global Optimization}, 27(4):367--374, 2003.

\bibitem{KFF09}
C.~Kanzow, I.~Ferenczi, and M.~Fukushima.
\newblock On the local convergence of semismooth {N}ewton methods for linear
  and nonlinear second-order cone programs without strict complementarity.
\newblock {\em SIAM Journal on Optimization}, 20(1):297--320, 2009.

\bibitem{Kawasaki88}
H.~Kawasaki.
\newblock An envelope-like effect of infinitely many inequality constraints on
  second-order necessary conditions for minimization problems.
\newblock {\em Mathematical Programming}, 41(1-3):73--96, 1988.

\bibitem{KS88}
H.~Kleinmichel and K.~Sch\"{o}nefeld.
\newblock Newton-type methods for nonlinearly constrained programming
  problems-algorithms and theory.
\newblock {\em Optimization}, 19(3):397--412, 1988.

\bibitem{Kong2011}
L.~Kong, L.~Tun{\c{c}}el, and N.~Xiu.
\newblock Equivalent conditions for {J}acobian nonsingularity in linear
  symmetric cone programming.
\newblock {\em Journal of Optimization Theory and Applications},
  148(2):364--389, 2011.

\bibitem{LFF16}
B.~F. Louren{\c{c}}o, E.~H. Fukuda, and M.~Fukushima.
\newblock Optimality conditions for nonlinear semidefinite programming via
  squared slack variables.
\newblock {\em {\normalfont To appear} in Mathematical Programming}, 2016.

\bibitem{NW99}
J.~Nocedal and S.~J. Wright.
\newblock {\em Numerical Optimization}.
\newblock Springer Verlag, New York, 1st edition, 1999.

\bibitem{N07}
D.~Noll.
\newblock Local convergence of an augmented {L}agrangian method for matrix
  inequality constrained programming.
\newblock {\em Optimization Methods and Software}, 22(5):777--802, 2007.

\bibitem{pataki_handbook}
G.~Pataki.
\newblock The geometry of semidefinite programming.
\newblock In H.~Wolkowicz, R.~Saigal, and L.~Vandenberghe, editors, {\em
  Handbook of Semidefinite Programming: Theory, Algorithms, and Applications}.
  Kluwer Academic Publishers, online version at
  \url{http://www.unc.edu/~pataki/papers/chapter.pdf}, 2000.

\bibitem{shapiro97}
A.~Shapiro.
\newblock First and second order analysis of nonlinear semidefinite programs.
\newblock {\em Mathematical Programming}, 77(1):301--320, 1997.

\bibitem{Sturm2000}
J.~F. Sturm.
\newblock Similarity and other spectral relations for symmetric cones.
\newblock {\em Linear Algebra and Its Applications}, 312(1-3):135--154, 2000.

\bibitem{SS08}
D.~Sun and J.~Sun.
\newblock {L}öwner's operator and spectral functions in {E}uclidean {J}ordan
  {a}lgebras.
\newblock {\em Mathematics of Operations Research}, 33(2):421--445, 2008.

\bibitem{SSZ08}
D.~Sun, J.~Sun, and L.~Zhang.
\newblock The rate of convergence of the augmented {L}agrangian method for
  nonlinear semidefinite programming.
\newblock {\em Mathematical Programming}, 114(2):349--391, 2008.

\end{thebibliography}
\end{document}